\documentclass[12pt]{amsart}
\usepackage{amssymb}
\usepackage{geometry}
\usepackage{amsmath}
\usepackage{amsfonts}
\usepackage[mathscr]{eucal}
\usepackage{amsthm}
\usepackage{bookmark}
\usepackage{enumerate}
\usepackage{extarrows}
\usepackage{bbm}
\usepackage{tikz}
\usepackage{changes}
\usepackage{color}
\usepackage{bbm}
\usepackage{verbatim}
\usepackage{graphicx}

\usepackage{courier}
\usepackage{helvet}         
\usepackage{courier}        
\usepackage{type1cm}        
\usepackage{multicol}        
\usepackage[bottom]{footmisc}

\date{}

\theoremstyle{plain}
\newtheorem*{thm}{Periodic Tiling Conjecture}

\newtheorem{theorem}{Theorem}[section]
\newtheorem*{theorem*}{Theorem}
\newtheorem*{question*}{Question}
\newtheorem{proposition}[theorem]{Proposition}
\newtheorem{lemma}[theorem]{Lemma}
\newtheorem{corollary}[theorem]{Corollary}

\theoremstyle{definition}

\newtheorem{question}[theorem]{Question}
\newtheorem{conjecture}[theorem]{Conjecture}

\numberwithin{equation}{section} \theoremstyle{plain}

\def\R{{\textbf{R}}}

\def\R{\mathbb R}
\def\N{\mathbb N}
\def\1{\mathbbm 1}
\def\Z{\mathbb Z}
\def\Q{\mathbb Q}

\begin{document}
\title{Periodicity of tiles in finite abelian  groups}
\author
{Shilei Fan}
\address
{Shilei FAN: School of Mathematics and Statistics, and Key Lab NAA--MOE, Central China Normal University, Wuhan 430079, China}
\email{slfan@ccnu.edu.cn}

\author
{Tao Zhang}
\address
{Tao ZHANG: Institute of Mathematics and Interdisciplinary Sciences, Xidian University, Xi'an 710071, China.}
\email{zhant220@163.com}
\date{}

\thanks{ S. L. FAN was partially supported by NSFC (grants No. 12331004  and No.  12231013). T. Zhang was partially supported by National Natural Science Foundation of China (Grant No. 12571357),
	 Natural Science Basic Research Program of Shaanxi  (Program No. 2025JC-YBMS-048), and Fundamental Research Funds for the Central Universities (Grant No. QTZX24082).}
\maketitle
\begin{abstract}
In this paper, we introduce the periodic tiling (PT) property for finite abelian groups. A finite abelian group is said to have the PT property if every non-periodic set that tiles the group by translation admits a periodic tiling complement. This notion extends the scope beyond groups with the Hajós property. We give a complete classification of cyclic groups possessing the PT property and  identify certain non-cyclic groups that enjoy the PT property but fail to satisfy the Hajós property.. As a byproduct, we obtain new families of groups for which the implication “Tile $\Longrightarrow$ Spectral” holds. Furthermore, for elementary $p$-groups with the PT property, by analyzing the structure of tiles, we prove that every tile is a complete set of representatives of the cosets of some subgroup.
\end{abstract}

\tableofcontents

\section{Introduction}

Let $G$ be a locally compact abelian group, and let $\Omega\subset G$ be a Borel measurable subset with $0<\mathfrak{m}(\Omega)<\infty$, where $\mathfrak{m}$ denotes the Haar measure on $G$ (sometimes denoted by $dx$). We say that  $\Omega$ is a
{\bf  tile}  of $G$ by translation if there exists a  set $T \subset G$ of translates
such that 
\[\sum_{t\in T} 1_\Omega(x-t) =1 \text{ for almost all }x\in G,\]
 where $1_A$ denotes the indicator function of a set $A$. Such a set $T$ is called a
{\bf tiling complement} of $\Omega$ and the pair $(\Omega,T)$ is called a {\bf tiling pair}. In this setting,
a tiling pair  $(\Omega,T)$ in $G$ means that $G=\Omega+T$ forms a factorization, that is each element $g\in G$ can be written uniquely in the form $g=\omega+t$ with $\omega\in \Omega$ and $t\in T.$
 For  a tile $\Omega$ in $G$, we denote by $\mathcal{T}_{\Omega}$ the set of  all tiling complements of $\Omega$.

In this paper, we introduce the concept of periodic tiling  property for finite abelian groups.  Let $G$ be a finite abelian group.  A nonempty subset $A\subset G$ is said  to be  {\bf periodic} if there exists a non-zero element   $g\in G$ such that $A+g=A$.
We say that an abelian group $G$ has the {\bf periodic tiling property} (abbreviated {\bf PT property}) if, for every tiling pair $(\Omega,T)$ of $G$,  either $\Omega$ is periodic, or $T$ can be replaced by a periodic one (that is, there exists a periodic set  $T'$ such that $(\Omega,T')$ is also a tiling pair).   

In 1938, G.  Haj\'os \cite{Haj38} reformulated  a well-known conjecture by H. Minkowski. The original conjecture stated that  if a Euclidean space of any dimension is tiled by lattice-positioned hypercubes, then some pairs must meet face-to-face. 
In 1941, Haj\'os  \cite{Haj41}  proved  the following statement, which is equivalent to Minkowski’s
conjecture.
\begin{theorem*} [Haj\'os]
 Let $G$ be a finite abelian group. If G can be written as
a direct sum of cyclic sets $A_i$; that is, $G = A_1+A_2+\cdots+A_n$, where $A_i$ is of
the form $A_i = \{0, a, 2a, 3a, . . . , ka\}$ with $a \in G$ and $0$ being the unit element
of $G$, then at least one of $A_i$ is a subgroup of G.
\end{theorem*}
This result brought attention to the factorization of finite abelian groups into factors that are not necessarily subgroups.
Following \cite[Page 5]{SS09book}, we say that a finite abelian group $G$ has the 
{\bf Haj\'os property} if, for every tiling pair $(\Omega,T)$ of $G$, 
either $\Omega$ or $T$ is periodic.  
In the literature, such groups are sometimes referred to as good groups (see, e.g., \cite{Sands1962}); here, we retain the term “Hajós property” to remain consistent with \cite{SS09book}.

In 1949, Hajós \cite{H49} posed the problem of classifying all finite abelian groups having this property. The classification was completed by Sands \cite{Sands1962}, who provided the following complete list of finite abelian groups with the Hajós property:
\begin{equation}\label{Hajos group}
	\begin{aligned}
		&\mathbb{Z}_{p^nq},\  \mathbb{Z}_{p^2q^2},\  \mathbb{Z}_{p^2qr},\  \mathbb{Z}_{pqrs},\  \mathbb{Z}_{p^3}\times\mathbb{Z}_2^2,\
		\mathbb{Z}_{p^2}\times\mathbb{Z}_{2}^3,\ \mathbb{Z}_{p}\times\mathbb{Z}_{4}\times\mathbb{Z}_{2},\\
		&\mathbb{Z}_{p}\times\mathbb{Z}_{2}^4,\ \mathbb{Z}_p\times\mathbb{Z}_q\times\mathbb{Z}_{2}^2,\ \mathbb{Z}_p\times\mathbb{Z}_{3}^2,\ \mathbb{Z}_{9}\times\mathbb{Z}_3,\ \mathbb{Z}_{2^n}\times\mathbb{Z}_2,\ \mathbb{Z}_4^2,\ \mathbb{Z}_p^2,
	\end{aligned}
\end{equation}
where $p,q,r,s$ are distinct primes and $n$ is any positive integer. The case $p=2$ is admitted in all types of groups given above. All these groups and all subgroups of them have the Hajós property.

It is evident that the Hajós property implies the PT property, but the converse is not true. For instance, every tile in $\Z_p^3$  admits a periodic tiling complement (see Proposition \ref{pro6}), whereas $\Z_p^3$ ($p\geq 5$) itself does not have the Hajós property.  

This naturally leads to the following question: in analogy with Hajós’ problem, can one give a complete description of all finite abelian groups having the PT property? We have obtained some partial results in this direction, including structural characterizations and several new examples of PT groups beyond the Hajós list. Nevertheless, a full classification remains out of reach. Further discussion on this problem will be given in Section~\ref{sec-con}.

As a starting point, we consider the following question:

\begin{question*}
	If all  proper subgroups of a group  have the Haj\'os property, does the group itself have  the  PT property?
\end{question*}

We show that the answer is negative, even for certain cyclic groups. In fact, we completely classify those cyclic groups that have the PT property.
\begin{theorem}\label{thm-cyc}
Finite cyclic groups  with the PT property are exactly  the  subgroups of the  groups
\[\mathbb{Z}_{p_{1}^np_2p_3\cdots p_k},\quad \mathbb{Z}_{p_{1}^2p_{2}^2},\]
where $p_1,p_2,p_3,\dots,p_k$ are distinct primes and $n\geq 1$ is any positive integer.
\end{theorem}

In proving that $\mathbb{Z}_{p_{1}^np_2p_3\cdots p_k}$ has the PT property, we use the fact that all the tiling complements of a non-periodic tile in $\mathbb{Z}_{p_{1}^n}$  share  a common period  $p_{1}^{n-1}$. This idea extends naturally to certain non-cyclic groups.
 
 We need the following properties for tiles in a finite abelian group $G:$ 
 \begin{itemize}
 \item A tile $\Omega$ is {\bf uniformly  periodic} if all its tiling complements have  a common period, i.e. $T+g=T$ for some  $g \in G\setminus\{0\}$ and for every $T\in \mathcal{T}_{\Omega}$.
 \item  A tile  $\Omega$ is {\bf dual uniformly periodic} if there exists a periodic tile $\Omega'$ such that $(\Omega',T)$ is a tiling pair for each $T\in\mathcal{T}_\Omega$. 
 \end{itemize}
 
We say that  a finite abelian group $G$  possess the {\bf uniformly periodic tiling property (UPT property)} if every  tile in $G$ is either uniformly periodic or dual uniformly periodic.  It is straightforward to verify that the UPT property is strictly stronger than the PT property (see Lemma~\ref{UPT>PT}).
Groups with the PT property can be constructed from those possessing the UPT property.

\begin{theorem}\label{thm:UPTPEX}
Let $G$ be a finite abelian group with the UPT property. Suppose $m$ is a square-free positive integer such that $\gcd(|G|, m) = 1$. Then the direct product group $G \times \mathbb{Z}_m$ has UPT property, and hence the PT property.
\end{theorem}

As an application of Theorem~\ref{thm:UPTPEX}, by establishing that the $p$-groups
\[
   \mathbb{Z}_{2}^{5},\quad \mathbb{Z}_{p}^{3},\quad \mathbb{Z}_{p^n} \times \mathbb{Z}_p,\quad \mathbb{Z}_4\times \mathbb{Z}_2^2
\]
possess the UPT property, we obtain further families of groups with the PT property, beyond the cyclic groups listed in Theorem~\ref{thm-cyc}.

\begin{theorem}\label{thm-noncyc}
Let \( G \) be a subgroup of one of the following groups:
\begin{align*}
   &\mathbb{Z}_{2}^{5},\quad \mathbb{Z}_{p}^{3},\quad \mathbb{Z}_{p^n} \times \mathbb{Z}_p,\quad \mathbb{Z}_4\times \mathbb{Z}_2^2,
\end{align*}
where \( p \) is a prime and \( n \) a positive integer.  
If \( m \) is square-free and \( \gcd(|G|, m) = 1 \), then the group \( G \times \mathbb{Z}_m \) has the PT property.
\end{theorem}
It is worth noting that  the cyclic group $\mathbb{Z}_{p^2q^2}$ has the PT property but  fails to have the UPT property. 
The UPT property of $G$ is crucial in ensuring that $G\times \mathbb{Z}_m$ inherits the PT property (see Proposition~\ref{thm:UPTPEX2} and Corollary~\ref{NUPT}).  
On the other hand, since $\mathbb{Z}_{p^2}\times \mathbb{Z}_{p^2}$ and $\mathbb{Z}_{p^2q^2}$ (with $p$ and $q$ distinct primes) fail to satisfy the UPT property (see Propositions~\ref{PropP2Q2} and \ref{prop:p2p2}), it follows that any group containing either $\mathbb{Z}_{p^2}\times \mathbb{Z}_{p^2}$ or $\mathbb{Z}_{p^2q^2}$ as a proper subgroup typically does not possess the PT property, apart from a few exceptional cases.

\begin{theorem}\label{thm-nonQG}
Let $G$ be a finite abelian group containing a proper subgroup isomorphic to $\mathbb{Z}_{p^2}\times\mathbb{Z}_{p^2}$ or $\mathbb{Z}_{p^2q^2}$, where $p,q$ are distinct primes and $q$ is odd.  
If $G$ is not isomorphic to one of the following exceptional groups:
\[
   \mathbb{Z}_{p^2}^2 \times \mathbb{Z}_2,\quad 
   \mathbb{Z}_{2p^2}\times\mathbb{Z}_{p^2},\quad 
   \mathbb{Z}_{9}^2\times\mathbb{Z}_3,\quad 
   \mathbb{Z}_{4}^2\times\mathbb{Z}_{2}^2,\quad 
   \mathbb{Z}_{4q^2}\times\mathbb{Z}_2,
\]
then $G$ does not have the PT property.
\end{theorem}
To further explore families of groups with the PT property, we observe that the Rédei property offers a useful approach.
Following \cite{Sza04}, a finite abelian group is said to have the \textbf{Rédei property} if, for every tiling pair $(\Omega, T)$ of $G$ with $0\in \Omega \cap T$, at least one of $\Omega$ or $T$ is contained in a proper subgroup of $G$.
Some groups, such as $\mathbb{Z}_{p_1 p_2 \cdots p_k}$ with $k\ge 6$, possess the PT property but lack the Rédei property, while others, such as $\mathbb{Z}_{p^2 q^3}$, have the Rédei property but not the PT property.
Nevertheless, we shall see that combining the Rédei property with the UPT property of all proper subgroups guarantees that the group itself has the PT property.


We also introduce a weaker variant: a finite abelian group has the \textbf{weak Rédei property} if, for every tile $\Omega$ of $G$ with $0 \in \Omega$ and $\langle \Omega \rangle = G$, there exists a tiling complement $T$ contained in a proper subgroup of $G$.

\begin{theorem}\label{thm:redei}
	Let $G$ be a finite abelian group with the  weak Rédei property. If every proper subgroup of $G$ has the UPT property, then $G$ has the PT property.
\end{theorem}
It is shown in \cite[Theorems  9.3.3, 9.3.4 and 9.3.10]{Sza04} that the groups  
\[
\mathbb{Z}_3^4, \quad \mathbb{Z}_9^2,\quad \mathbb{Z}_2^6
\] 
have the Rédei property. As an immediate consequence of Theorem~\ref{thm:redei}, all of these groups also possess the PT property.

It is known that $\Z_4 \times \Z_4$ does not have the UPT property.  
Nevertheless, we shall prove that $\Z_8 \times \Z_4$ and $\Z_4 \times \Z_4 \times \Z_2$ both enjoy the PT property, relying on their Rédei property (see \cite[Theorem~9.3.6 and Corollary~9.3.2]{Sza04}).

\begin{theorem}\label{thm:84442}
The groups $G = \Z_8 \times \Z_4$ and $\Z_4^2 \times \Z_2$ have the PT property.
\end{theorem}

We now turn to the structure of tiles in groups with the PT property. In \cite{SS09book}, it was shown that every tile in a group with the Hajós property admits an \emph{ascending chain structure} (see Section~\ref{sec-structure} for the definition). In this paper, we prove that if every tile of a group has an ascending chain structure, then the group necessarily has the PT property. Moreover, we establish the following result.
 \begin{theorem}\label{theorem-structure}
	Let $G$ be a finite abelian  group. 
	\begin{enumerate}[{\rm (1)}]
		\item If $G$ and  all its subgroups have the PT property, then any tile  of $G$ have the ascending chain structure.
		\item If  any tile  of $G$ has ascending chain structure, then $G$ has the PT property.
	\end{enumerate}
\end{theorem}
As a consequence of Theorem \ref{theorem-structure}, we can characterize the tiles in elementary $p$-groups that have the PT property.
\begin{theorem}\label{cor-finitgroup}
	Suppose    $\Omega$ is a tile of the   elementary $p$-group $\mathbb{Z}_{p}^n$.
	\begin{enumerate}[{\rm (1)}]
		\item Case  $p\geq 3$: If $\mathbb{Z}_{p}^n$ has the PT property, then   $(\Omega,T)$  is a tiling pair  for some  subgroup  $T$.
		\item Case $p=2$: If  $\mathbb{Z}_{2}^n$  and  $\mathbb{Z}_{2}^{n-1}$ have PT property, then   $(\Omega,T)$  is a tiling pair for some subgroup   $T$.
	\end{enumerate}
\end{theorem}

The second motivation of this paper is to explore an application of the PT property to Fuglede’s conjecture. In harmonic and functional analysis, a fundamental question asks whether a geometric property of sets (tiling) and an analytic property (spectrality) are always two sides of the same coin. This question was initially posed by Fuglede \cite{F74} for finite-dimensional Euclidean spaces, stemmed from a question of Segal on the commutativity of certain partial differential operators.

\begin{conjecture}\label{Fuglede}
	A Borel set $\Omega\subset \mathbb{R}^d$ of positive and finite Lebesgue measure is a spectral set if and only if it is a tile.
\end{conjecture}

The original Fuglede conjecture (Conjecture \ref{Fuglede}) has been disproven in its full generality for dimensions 3 and above for both directions \cite{FMM06,KM2006,KM06,M05,T04}. This means that neither implication (tiling implies spectral and vice versa) holds true in these higher dimensions.  However, the connection between tiling and spectral properties remains an active area of research, particularly in lower dimensions. The conjecture is still open for the one-dimensional and two-dimensional cases ($\R$ and $\R^2$). There might be a deeper relationship to be discovered in these simpler settings (see \cite{DL2014} for a focused look on $\mathbb{R}$).

Despite the general counterexamples, the conjecture has been successfully proven for convex sets in higher dimensions. Iosevich, Katz, and Tao  \cite{IKT03}  initiated this progress in 2003 by demonstrating the validity of the conjecture for convex sets in $\mathbb{R}^2$. This result was later extended to $d = 3$ (three dimensions) by Greenfeld and Lev \cite{GL17} in 2017. Finally, Lev and Matolcsi \cite{LM21} achieved a major breakthrough in 2021 by proving the conjecture for all convex sets in $\mathbb{R}^n$ ($n\ge3$).

There has been a growing interest in extending the Fuglede conjecture beyond the realm of Euclidean spaces. Fuglede himself hinted at the possibility of exploring the conjecture in different settings. This has led to a more general version of the conjecture applicable to locally compact abelian groups.

The generalized Fuglede conjecture has been proved for different groups, particularly within the realm of finite abelian groups.  These successes include groups like $\mathbb{Z}_{p^{n}}$ \cite{L02}, $\mathbb{Z}_{p}^{d}$ ($p=2$ and $d\le6$; $p$ is an odd prime and $d=2$; $p=3,5,7$ and $d=3$) \cite{AABF17,FMV2022,FS20,IMP17},  $\mathbb{Z}_{p}\times\mathbb{Z}_{p^{n}}$ \cite{IMP17,S20,Z2022},    $\mathbb{Z}_{p}\times\mathbb{Z}_{pq}$ \cite{KS2021}  and $\mathbb{Z}_{pq}\times\mathbb{Z}_{pq}$ \cite{FKS2012},  $\mathbb{Z}_{p^{n}q^{m}}$ ($p<q$ and $m\le9$ or $n\le6$; $p^{m-2}<q^{4}$) \cite{KMSV20,M21,MK17}, $\mathbb{Z}_{pqr}$ \cite{S19}, $\mathbb{Z}_{p^{2}qr}$ \cite{Somlai21}, $\mathbb{Z}_{p^nqr}$ \cite{Z2023} and $\mathbb{Z}_{pqrs}$ \cite{KMSV2012}, where $p,q,r,s$ are distinct primes. Fan et al. \cite{FFLS19,FFS16} established the validity of the conjecture for the field $\Q_p$ of $p$-adic numbers, presenting the first example of an infinite abelian group where Fuglede's conjecture holds.

Building on the notation from \cite{DL2014}, we define $S-T(G)$ (respectively, $T-S(G)$) to indicate whether the ``Spectral $\Rightarrow$ Tile" (respectively,  ``Tile $\Rightarrow$ Spectral") direction of Fuglede's conjecture holds in group $G$. In this context, the following relationships are proved in \cite{DL2013,DL2014}:
\[T-S(\mathbb{R}) \Leftrightarrow T-S(\mathbb{Z}) \Leftrightarrow T-S(\mathbb{Z}_N),\ \forall N \in \N,\]
and 
\[S-T(\mathbb{R}) \Rightarrow S-T(\mathbb{Z}) \Rightarrow S-T(\mathbb{Z}_N),\ \forall N\in\N.\]
These relations highlight the critical role of finite cyclic groups, denoted by $\mathbb{Z}_N$ here, in understanding Fuglede's conjecture for the real numbers, $\mathbb{R}$.

 Focusing on the ``Tile $\Rightarrow$ Spectral" direction, Łaba \cite{L02} established $T-S(\mathbb{Z}_{p^nq^m})$ for distinct primes $p$ and $q$. Later, Łaba and Meyerowitz proved $T-S(\mathbb{Z}_n)$ for square-free integers $n$ (see the discussion in Tao's blog \cite{T2011} or \cite{S19}). More recently, Malikiosis \cite{M21} extended this result to groups of the form $\mathbb{Z}_{p_{1}^{n}p_2\cdots p_k}$, where $p_1, p_2, \dots, p_k$ are distinct primes. Łaba and Londner \cite{LL2022, LL22, LL23, LL24} introduced a new tool for studying tiles in cyclic groups, and in particular, they proved $T-S(\mathbb{Z}_{p_1^2p_2^2p_3^2p_4\dots p_k})$ and $T-S(\mathbb{Z}_{p_1^{n}p_2^mp_3p_4\dots p_k})$, where $p_1,\dots,p_k$ are distinct primes, and $n,m$ are positive integers. From the perspective of universal spectra, Zhou \cite{Zhou2024} obtained a similar result.
In this paper, we prove the following.
\begin{theorem}\label{PT--TS}
	Let $G$ be a finite abelian group. Assume that $G$ and all its subgroups have the PT property. If $\Omega$ tiles $G$ by translation, then $\Omega$ is a spectral set.
\end{theorem}

Combining Theorems~\ref{thm-noncyc}, \ref{thm:redei}, \ref{thm:84442} and \ref{PT--TS}, we obtain the following families of groups in which the ``Tile $\Rightarrow$ Spectral" implication holds.
\begin{corollary}
	\begin{enumerate} [{\rm (1)}]
		\item In each of the groups listed in \ref{Hajos group}, every tile is a spectral set.
	\item Let \( G \) be a subgroup of one of the following groups:
		\begin{align*}
			&\mathbb{Z}_{2}^{5},\ \mathbb{Z}_{p}^{3},\ \mathbb{Z}_{p^n} \times \mathbb{Z}_p,\ \mathbb{Z}_4\times \mathbb{Z}_2^2,
		\end{align*}
		where \( p \) is a prime and \( n \) is a positive integer. If \( m \) is square-free and \( \gcd(|G|, m) = 1 \), then every tile is spectral in  \( G \times \mathbb{Z}_m \).
		\item Every tile is spectral in the groups:
		\begin{align*}
			&\mathbb{Z}_{2}^{6},\ \mathbb{Z}_{3}^{4},\ \mathbb{Z}_{9}^2,\ \mathbb{Z}_8\times\mathbb{Z}_4,\ \mathbb{Z}_4^2\times\mathbb{Z}_2.
		\end{align*}
	\end{enumerate}
\end{corollary} 

It is known \cite{AABF17,FS20} that there exist spectral sets which are not tiles in   $\Z_{p}^{4}$ with $p\geq 3$ and also in $\Z_{2}^{10}$. However,   no counterexample is known among  $p$-groups for the ``Tile $\Longrightarrow$ Spectral" direction.

Motivated by the above discussion, it is natural to introduce the following terminology.  
\begin{itemize}
	\item A finite abelian group \(G\) is called a {\bf T–S group} if every translational tile in \(G\) is also a spectral set, i.e., Fuglede’s conjecture holds in the direction ``Tile \(\Rightarrow\) Spectral."  
	\item A finite abelian group \(G\) is called an {\bf S–T group} if every spectral set in \(G\) is a tile, i.e., Fuglede’s conjecture holds in the direction ``Spectral \(\Rightarrow\) Tile."  
\end{itemize}

With this terminology, many previously studied groups turned out to be both T–S and S–T groups. On the other hand, there are examples (see \cite{AABF17, FS20}) showing that \(\Z_{p}^{4}\) with \(p\geq 3\) and \(\Z_{2}^{10}\) fail to be S–T groups, as they admit spectral sets that are not tiles.

Our main result highlights a fundamentally new phenomenon:  
{\bf \(\mathbb{Z}_3^4\) is a T–S group but not an S–T group. }This is, to the best of our knowledge, the first finite abelian group for which Fuglede’s conjecture holds in exactly one direction. In all previously known cases, the conjecture was either validated in both directions or refuted in some direction. Hence, the case of \(\mathbb{Z}_3^4\) provides the earliest concrete evidence of a genuinely one-sided validity of Fuglede’s conjecture, making it particularly remarkable. 
As far as we know, {\bf  no example of a finite abelian group that is S–T but not T–S has been discovered. }   

\begin{figure}[h]
\centering
\begin{tikzpicture}[scale=1.5]

\filldraw[fill=blue!30, draw=black, opacity=0.6] (0,0) circle (2cm);
\filldraw[fill=green!30, draw=black, opacity=0.6] (2.8,0) circle (2cm);

\node[align=center] at (-1.2,0) {\textbf{T--S only}\\(e.g.\ $\mathbb{Z}_3^4$)};

\node[align=center] at (4.0,0) {\textbf{S--T only}\\(no known example)};

\node[align=center] at (1.4,0) {\textbf{Both}\\T--S \& S--T\\(many known groups)};


\end{tikzpicture}
\caption{Venn diagram of T--S and S--T groups}
\end{figure}


It was conjectured in \cite{S19} that the implication ``Tile $\Longrightarrow$ Spectral" holds for all $p$-groups. For elementary  $p$-groups, one possible approach is to analyze the periodic structure of tiles. 
As a corollary of Theorem~\ref{cor-finitgroup}, we have:
\begin{corollary}
Let     $\Omega$  be  a tile in  the   elementary $p$-group $\mathbb{Z}_{p}^n$.
\begin{enumerate}[{\rm (1)}]
\item Case  $p\geq 3$: If $\mathbb{Z}_{p}^n$ has the PT property, then  $\Omega$ is a spectral set 
\item Case $p=2$: If  $\mathbb{Z}_{2}^n$  and  $\mathbb{Z}_{2}^{n-1}$ have the PT property, then  $\Omega$ is a spectral set. 
\end{enumerate}
\end{corollary}

This paper is organized as follows.
Section~\ref{sec-pre} reviews the basic notions of the Fourier transform and translational tilings in finite abelian groups.
In Section~\ref{sec-ptag}, we investigate the PT property in relation to subgroups of a given group.
Section~\ref{sectionUPTP} introduces the UPT property and presents the proof of Theorem~\ref{thm:UPTPEX}.
In Section~\ref{sec-UPTP-group}, we examine the UPT property for several specific classes of groups.
Section~\ref{sec-proofs} contains the proofs of Theorems~\ref{thm-cyc}, \ref{thm-noncyc} and  \ref{thm-nonQG}.
In Section~\ref{sec-Red-PT}, we explore the relationship between the Rédei property and the PT property, and prove Theorems~\ref{thm:redei} and \ref{thm:84442}.
Section~\ref{sec-structure} analyzes the structure of tiles in PT groups, providing the proofs of Theorems~\ref{theorem-structure} and \ref{cor-finitgroup}.
Section~\ref{sec-PT--TS} is devoted to the proof of Theorem~\ref{PT--TS}.
Finally, Section~\ref{sec-con} summarizes and concludes the paper.

\section{Preliminaries}\label{sec-pre}
In this section, we briefly review some basic definitions, the Fourier transform, and the equivalent characterization of tiling on finite abelian groups.

Let $G$ be a finite abelian group, and let $\mathbb{C}$ be the set of complex numbers. A character on $G$ is a group homomorphism $\chi: G \to \mathbb{C}\setminus\{0\}$. The dual group of a finite abelian group $G$, denoted as $\widehat{G}$, is the character group of $G$. For  a  subset $A\subset G$, define 
\[\chi(A):=\sum_{x\in A}\chi (x). \]
 

Any finite abelian group $G$ can be written as  $\Z_{n_1} \times \Z_{n_2} \times \cdots \times \Z_{n_s}$.  For $g=(g_1,g_1, \dots, g_s)\in G$, denote by $\chi_g$ the character 
\[
\chi_{g}(x_1, \dots, x_s) = e^{2\pi i \sum_{j=1}^{s} \frac{x_jg_j}{n_j}}.
\]
For $g,h\in G$, it is clear that 
\[\chi_{g+h}(x)=\chi_{g}(x)\cdot \chi_{h}(x).\]
The dual group  $\widehat{G}$ is isomorphic to itself, i.e.
\[\widehat{G}\cong\widehat{\Z}_{n_1} \times \widehat{\Z}_{n_2} \times \cdots \times \widehat{\Z}_{n_k}\cong\Z_{n_1} \times \Z_{n_2} \times \cdots \times \Z_{n_s}.\]

For  two  finite abelian groups $G_1,G_2$, let $G$ be their  product $G_1\times G_2$. It is known that 
$\widehat{G}\cong\widehat{G_1}\times \widehat{G_2}\cong G_1\times G_2$, and each character in $\widehat{G}$ can be written as \[\chi_{(g_1,g_2)}(x_1,x_2)=\chi_{g_1}(x_1)\chi_{g_2}(x_2),\] 
where $g_1\in G_1$ and $g_2\in G_2$.

\subsection{Fourier Transform}
The Fourier transform on $G$ is a linear transformation that maps a function $f: G \rightarrow \mathbb{C}$  to a function $\widehat{f}: \widehat{G} \rightarrow \mathbb{C}$ defined as follows:
\[\widehat{f}(g) = \sum_{x \in G} f(x) \cdot \chi_g(-x)\]
where $\chi_g$ is the character of $G$  corresponding to $g$, and $f(x)$ is the value of the function $f$ at the element $x$ in $G$.

For $A \subset G$, denote by  
\[\mathcal{Z}_{A}:=\{x\in \widehat{G}: \widehat{\mathbbm{1}_{A}}(x)=0\}\]
the set of zeros  of the Fourier transform of the indicator function $\mathbbm{1}_A$. 
The set $\mathcal{Z}_{A}$, determined by the vanishing of certain sums of roots of unity, reflects structural properties of 
$A$.
The following lemma, due to R\'edei \cite{Redei54}, will be useful.
\begin{lemma}\label{lemma-pr}
	Let $p$ be a prime and $\zeta=\zeta_{p^n}$ be a primitive $p^n$-th root of unity. Suppose $c=c_{p^n-1}\zeta^{p^n-1}+c_{p^n-2}\zeta^{p^n-2}+\cdots+c_1\zeta+c_0$, where $c_i\in\mathbb{Z}$, $0\le i\le p^n-1$. Then $c=0$ if and only if $c_i=c_j$ for any $i,j$ with $i\equiv j\pmod{p^{n-1}}$.
\end{lemma}

 As consequences of Lemma \ref{lemma-pr}, the following lemmas provide a sufficient condition for a set in $G= \Z_{p^n}$ or  $\Z_{p^n} \times H$ to be periodic.
 \begin{lemma}\label{lem:periodic-1}
 Let $A\subset \Z_{p^n}$. If $ 1 \in \mathcal{Z}_A$, then $A+p^{n-1}=A$.
 \end{lemma}
\begin{proof}
Since $ 1 \in \mathcal{Z}_A$, then 
\[\sum_{x\in A} e^{2\pi i \frac{x}{p^n}}=0.\]
By Lemma \ref{lemma-pr}, if $x\in A$, then $x+jp^{n-1}\in A$ for any $j\in \{0, 1, \cdots, n-1\}.$ Hence \[A+p^{n-1}=A.\]
\end{proof}

 \begin{lemma}\label{lem:periodic}
 Let $H$ be a finite abelian group  and let  $A\subset G=\Z_{p^n}\times H$.  If $(1, h) \in \mathcal{Z}_A$ for each
$h \in  H$, then $A+(p^{n-1},0)=A$.
 \end{lemma}
 \begin{proof}
Let $A_h=\{x\in \Z_{p^n}: (x,h)\in A\}$. Then we have \[A=\bigcup_{h\in H}A_h\times\{h\}.\]
Since for each $\alpha\in H$,  $(1,\alpha)\in \mathcal{Z}_A$, it follows that 
\begin{align} \label{eq:zero1}
\sum_{(x,h)\in A}e^{2\pi i \frac{x}{p^n}} \chi_{\alpha}(h)=\sum_{h\in H}\chi_{\alpha}(h) \sum_{x\in A_h} e^{2\pi  i\frac{x}{p^n}}=0.
\end{align}
Let $X_h= \sum_{x\in A_h} e^{2\pi i \frac{x}{p^n}}$ and   let $\mathbf{X}=(X_h)_{h\in H}$ be the row vector with elements $X_h$.
Let   $M_H=(\chi_{\alpha}(h))_{\alpha,h \in H}$  be the Fourier matrix  of $H$. By Equality \eqref{eq:zero1}, we obtain the following system of linear equations

\[
M_H \cdot  \mathbf{X}
=
\begin{pmatrix}
0 \\
0 \\
\vdots \\
0
\end{pmatrix}.
\]
The coefficient matrix $M_H$ is of full rank, therefore $X_h=0$ for each $ h\in H$, which implies that 
\[A_h+p^{n-1}=A_h.\]
Therefore $A+(p^{n-1},0)=A$.
 \end{proof}

\subsection{Equivalent characterization for tiling pairs} \label{sect-2.1}

   Assume that $(\Omega,T)$ is a tiling pair of $G$.  Recall that this means  that \[G=\Omega+T\] forms a factorization, which is  equivalent to
\begin{equation}\label{tiling}
\mathbbm{1}_{\Omega}*\mathbbm{1}_{T}\equiv 1,
\end{equation}
where $\mathbbm{1}_E$ is the indicator
function of $E$. For a finite set $E$, denote by $|E|$ the cardinality of $E$.
By taking Fourier transform, (\ref{tiling}) is equivalent to
\begin{equation}\label{Fouriertiling}
\widehat{\mathbbm{1}_{\Omega}} \cdot \widehat{\mathbbm{1}_{T}}=|G|\cdot \delta_{0}.
\end{equation}
We have the following equivalent conditions for a tiling pair (see \cite{S20}, \cite[Lemma 2.1]{SS09book}).
\begin{lemma}\label{lem3p2}
	Let $\Omega,T$ be subsets in  a  finite subgroup $G$. Then the following statements are equivalent:
	\begin{enumerate}[{\rm(a)}]
		\item $(\Omega,T)$ is a tiling pair.
		\item $(T,\Omega)$ is a tiling pair.
		\item $(\Omega+g,T+h)$ is a tiling pair.
		\item $|\Omega|\cdot|T|=|G|$ and $(\Omega-\Omega)\cap(T-T)=\{0\}$.
		\item $|\Omega|\cdot|T|=|G|$ and $\mathcal{Z}_{\Omega}\cup\mathcal{Z}_{T}=G\backslash\{0\}$.
	\end{enumerate}
\end{lemma}

Recall that $\mathcal{T}_\Omega$ is  the set of tiling complements of $\Omega$.
For a subset $A\subset G$ and integer $k$, let $kA=\{k\cdot a: a\in A\}$.  The following  lemma states that a tiling complement $T$  can be replaced by $kT$ if $k$ and $|T|$ are relatively prime.
\begin{lemma}\label{lemma-replace}$($\cite[Theorem 3.17]{SS09book}$)$
Assume that   $(\Omega,T)$  is a tiling pair of  a finite abelian group $G$. 
If  $(k, |T|)$=1, then $kT\in \mathcal{T}_\Omega$.  
\end{lemma}

\begin{corollary} \label{cor-mtimes}
Let $G$ be a finite abelian group  and  let  $m$ be an integer. If   $(m,|G|)=1$, then $\Omega$ tiles $G$ by translation if and only if $m \Omega$ tiles $G$ by translation. Moreover, $\mathcal{T}_\Omega=\mathcal{T}_{m\Omega}.$
\end{corollary}

\section{Periodic tiling in abelian groups}\label{sec-ptag}
It is known that the subgroups of a group with  the Haj\'os property also have the Haj\'os property. 
This naturally leads to the following question.
\begin{question}
If a finite abelian  group has the PT property, do all its subgroups have the PT property?
\end{question}
We prove that  this is true under the additional condition that the groups are not $2$-group.
\begin{theorem}\label{theorem-subgroup1}
Let $G$ be a finite abelian group that is not a $2$-group. If $G$ has the PT property,  then all its subgroups  have the PT property.
\end{theorem}
For $2$-groups, it is unclear how the periodic tiling property of $(\mathbb{Z}_2)^{n+1}$ implies that of $(\mathbb{Z}_2)^n$. However,  we obtain a partial result.  
\begin{theorem}\label{theorem-subgroup2}
 Let  $n\ge2$ and $H$ be a $2$-group.
\begin{enumerate}[{\rm(1)}]
 \item If $H\times\mathbb{Z}_{2^n}$ has the PT property, then $H$ and  $H\times\mathbb{Z}_{2^{n-1}}$  has the PT property.
 \item If  $H\times\mathbb{Z}_{2}^n$ has the PT property, then  $H$  has the PT property.
 \end{enumerate}
 \end{theorem}
Note that every subgroup of $\mathbb{Z}_2^5$ possesses the Haj\'os property, and therefore also the PT property.  As a consequence of  Theorems \ref{theorem-subgroup1} and  \ref{theorem-subgroup2}, we have the following corollary. 
\begin{corollary} 
 Let $G$ be a  finite abelian group having the  PT property.  If  the rank of $G$ (the smallest cardinality of a generating set) is at most $5$, then all its  subgroups  have the PT property.
\end{corollary}

\subsection{Construction of tiles from  the tiles in subgroups}
To prove Theorems \ref{theorem-subgroup1} and \ref{theorem-subgroup2}, the first step is to construct  non-periodic tiles  based on tiles in the subgroups.

\begin{lemma}\label{prop:tile}
 Let $G = H \times S$ be the direct product of two finite abelian groups, $H$ and $S$. Suppose $(\Omega, T)$ is a tiling pair for the group  $H$ and  identify $H$ with the subset $H \times \{0\} \subset G$. For any subset $K \subset G$ of the form \[K=\{(h_s, s): s \in S, h_s \in H\}, \] 
the pair $(\Omega+K, T)$ forms a tiling pair of $G$.
Furthermore, if neither $\Omega$ nor $K$ is periodic, then $\Omega+K$ is also non-periodic.
\end{lemma}
 \begin{proof}
 Represent elements in $G$ as the form $(h,s)$,  with $h\in H$ and $s \in S$.
  Remark that if $H$ is considered as a subgroup of $G$,  the elements in $H$ take the form $(h, 0)$.



Let \[\widetilde{\Omega} = \Omega + K.\]
 Note that $H \cap (H+(h_s,s)) = \emptyset$  if $s\neq 0$ and $\Omega+T+(h_s, s)=H+(0,s)$ for $(h_s,s)\in K$. Hence, $(\widetilde{\Omega},T)$ forms a  tiling pair of $G$.

Suppose that  both $\Omega$ and $K$ are not periodic. We shall prove by contradiction that $\widetilde{\Omega}$ is not periodic.

Assume that $\widetilde{\Omega} + (h_0, s_0) = \widetilde{\Omega}$ for some $(h_0, s_0) \in G\setminus\{0\}$. It follows that
\[\Omega+ K + (h_0, s_0) = \Omega+ K.\]
 If $s_0 = 0$, then $\Omega+ (h_0, 0) = \Omega$, implying $(h_0, s_0) = (0, 0)$.
Hence, $s_0 \neq 0$  and  we have the following relations:
\begin{align}\label{eq-h1}
\forall s \in S, \quad \Omega+ (h_s, s)+ (h_0, s_0) = \Omega+ (h_{s+s_0}, s+s_0).
\end{align}
As $\Omega$ is non-periodic, Equality~\eqref{eq-h1}  implies $h_{s+s_0}=h_s+h_0$. Hence,
\[K+(h_0,s_0)=\{(h_{s+s_0}, s+s_0): s \in S, h_s \in H\}=K,\]
which implies $K$ is periodic,   leading to a contradiction. Therefore, $\widetilde{\Omega}$ is non-periodic.
\end{proof}

\begin{lemma}\label{prop:2.1}
Consider a finite abelian group $S$, and let $H = S \times \Z_{p^n}$ and $G= S \times \Z_{p^{n+1}}$ where $n\geq 1$. Assume  $(\Omega, T)$ is a tiling pair of $H$, and identify $H$ with the subset $\{(s,pj)\in G:  s\in S, j\in \{0,1, \ldots, p^n-1 \}\}$ of $G$. For any $h\in H$, let $K \subset G$ be of the form 
\[K=\big\{(0, i)+h: i \in \{0,1, \ldots, p-1 \}\}.\] 
Then $ (\Omega + K,  T)$ forms a tiling pair of $G$. Moreover, if $\Omega$ is not periodic in $H$, then $\Omega + K$ is not periodic in $G$.
\end{lemma}

\begin{proof}
Let $\widetilde{\Omega} = \Omega + K.$ Note that $H \cap (H+(0,i)) = \emptyset$ for any $1\leq i\leq p-1$ and $\Omega + T + (0, i) + h = H + (0,i)$ for $(0,i)+h \in K$. Hence, we have $G = \widetilde{\Omega} + T$.

Assume that $\Omega$ is not periodic in $H$ and $\widetilde{\Omega} + (s,i) = \widetilde{\Omega}$ for some nonzero element $(s,i) \in G$. Write $i = i_0 + pi_1$, where $i_0 \in \{0,\ldots, p-1\}$.

If $i_0 = 0$, then $\Omega + (s, pi_1) = \Omega$. So, it follows that $(s, i) = (0, 0)$. Hence, $i_0 \neq 0$ and 
\begin{align}
&\Omega + (s, pi_1 + i_0) = \Omega + (0, i_0),\label{eq1}\\
&\Omega + (0, p-i_0) + (s, pi_1 + i_0) = \Omega.\label{eq2}
\end{align}
Equation~(\ref{eq1}) implies $i_1 = 0$ and Equation~(\ref{eq2}) implies $i_1 = -1$, which is a contradiction. Hence, $\widetilde{\Omega}$ is non-periodic.
\end{proof}

\subsection{Proof of Theorem \ref{theorem-subgroup1} }
Let $H$ be  a subgroup of $G$. Assume that $H$ does not have the PT property. 
For simplicity, we can focus on the case where $[G:H]=p$, where $p$ is a prime. This is because we can use induction on the index to prove the result for any finite index subgroup. Hence, Theorem  \ref{theorem-subgroup1}  is a consequence of the following  Lemmas \ref{lem:productextension}, \ref{thm:twoextension} and \ref{Thm-3.7}.

\subsubsection{Induction from $H$ to $H\times \Z_p$}

In this subsection, we study the behavior of the PT property under direct products.  
Our main goal is to show that if a finite group $H$ does not have the PT property, then the product $H\times S$ also fails to have the PT property for various finite abelian groups $S$.

We first prove a general result for groups $S$ with $|S|\ge 3$ (Lemma~\ref{lem:productextension}), which relies on a combinatorial fact (Lemma~\ref{lem:S-S}).  
As a direct consequence, the case $S=\Z_p$ with $p$ an odd prime follows immediately (Corollary~\ref{thm:primeextension}).  

The case $S=\Z_2$ is treated separately in Lemma~\ref{thm:twoextension}, since the previous combinatorial method does not apply. Here we assume that $H$ is not a $2$-group, and a different argument is required to establish the result.

Together, these lemmas show that the lack of the PT property is generally preserved under direct products with cyclic groups of prime order, with special care needed for the $2$-group case.

%

\begin{lemma}\label{lem:productextension}
Let $H$ be a finite group which does not have the PT property.  Then, for any finite  abelian group $S$ with  $|S|\geq 3$, the group $G=H\times S$ does not have the PT property.
\end{lemma}

\begin{lemma}\label{lem:S-S}
Let \(S\) be a finite abelian group with \(|S|\ge 3\). Then 
\[
(S\setminus\{0\})-(S\setminus\{0\})=S,
\]
\end{lemma}
\begin{proof}

Let \(S\) be a finite (additive) group with \(|S|\ge 3\). 
Choose two distinct  elements \(a,b\in S\setminus\{0\}\). 
We show that every \(s\in S\) can be written as \(s=x-y\) with \(x,y\in S\setminus\{0\}\).

For a given \(s\in S\), consider the two elements \(s+a\) and \(s+b\).
If \(s+a\neq0\), then set \(x=s+a\) and \(y=a\); clearly \(x,y\in S\setminus\{0\}\) and \(s=x-y\).
If \(s+a=0\), then \(s=-a\). Since \(a\neq b\), we have \(s+b\neq0\). Thus in this case set \(x=s+b\) and \(y=b\), giving \(s=x-y\) with \(x,y\in S\setminus\{0\}\).
Therefore every \(s\in S\) lies in \((S\setminus\{0\})-(S\setminus\{0\})\).
\end{proof}

\begin{proof}[Proof of Lemma  \ref{lem:productextension}]
Note that any element of $G$ can be represented as $(h,s)$, where $h\in H$ and $s\in S$.
Let $(\Omega, T)$ be a tiling pair of  $H$. Assume  $\Omega$ is not periodic and $T$ cannot be replaced by a periodic set. For $0\neq h_0\in H$, let \[K=\{(0,0)\}\cup \{(-h_0,s): s\in S\setminus \{0\} \}.\]
Note that $K$ is not periodic.
   Let  $\widetilde{\Omega}=(\Omega,0)+K.$ Then, by  Lemma \ref{prop:tile}, $\widetilde{\Omega}$ is not periodic.

Suppose  $G$ has the PT property.  Then there exists a periodic set $\widetilde{T}$ such that $G=\widetilde{\Omega}+\widetilde{T}$. For any  distinct $(h_1,s_1),(h_2,s_2)\in\widetilde{T}$, we have
  \[\big(\widetilde{\Omega}+(h_1,s_1)\big)\cap \big(\widetilde{\Omega}+(h_2,s_2)\big)=\emptyset.\]
  
  If $s_1=s_2$, then 
\[
(\Omega+h_1)\cap(\Omega+h_2)=\emptyset.
\]
If $s_1\ne s_2$, then by Lemma~\ref{lem:S-S} there exist $s_3,s_4\in S\setminus\{0\}$ such that 
\[
s_1+s_3=s_2+s_4.
\]
Hence,
\[
(\Omega-h_0+h_1,s_1+s_3)\cap(\Omega-h_0+h_2,s_2+s_4)=\emptyset,
\]
which implies
\[
(\Omega+h_1)\cap(\Omega+h_2)=\emptyset.
\]
Define
\[
T'=\{\,h : (h,s)\in\widetilde{T}\ \text{for some }s\in S\,\}.
\]
Then we obtain
\[
\Omega+T'=H.
\]

Note that $(0,s)$, $s\ne0$ cannot be a period of $\widetilde{T}$ since $s\in S\backslash\{0\}-S\backslash\{0\}$.
If $(h,s)$ is a periodic of $\widetilde{T}$, then $(h,0)$ is also a period of $T'$, which is a contradiction.
\end{proof}

\begin{corollary}\label{thm:primeextension}
Let $H$ be a finite group which does not have the PT property.  Then, for any odd prime $p$, the group $G=H\times\mathbb{Z}_p$ does not have the PT property.
\end{corollary}

\begin{lemma}\label{thm:twoextension}
Let $H$ be a finite abelian group that is not a $2$-group. Assume $H$ does not have the PT property. Then the group $G = H \times \mathbb{Z}_2$ does not have the PT property.
\end{lemma}
\begin{proof}
Any element of $G$ can be represented by $(h,i)$, where $h \in H$ and $i \in \mathbb{Z}_2$. All elements of $H$ have the form $(h,0)$.

Let  $(\Omega,T)$ be a tiling pair of $H$. Assume  $\Omega$ is not periodic and $T$ can not be replaced by a periodic set. Let $h_0$ be any element in $H$ with order an odd  $p$. 
Let
\[\widetilde{\Omega} = \Omega \cup (\Omega + (-h_0,1)).\]
By Lemma \ref{prop:tile}, it follows that $G = \widetilde{\Omega} + T$ is a factorization, and $\widetilde{\Omega}$ is not periodic.

Suppose that $G$  has the PT property. Then there exists a periodic subset $\widetilde{T}$ such that
\[ G = \widetilde{\Omega} + \widetilde{T}.\]

If $\widetilde{T} + (h',1) = \widetilde{T}$ with $\text{ord}(h',1) = 2$, then there exists $T' \subset H$ such that $\widetilde{T} = T' + \{(0,0),(h',1)\}$. Here $\text{ord}(\cdot)$ denotes the order of the element. Since $G = \widetilde{\Omega} + \widetilde{T}$, we have
\[H = \Omega + T' + \{(0,0), (h' - h_0,0)\}.\]
By Lemma \ref{lemma-replace}, it follows that
\[H = \Omega + T' + p\{(0,0), (h' - h_0,0)\},\]
which implies
\[H = \Omega + T' + \{(0,0), (h',0)\}.\]
However, $T' + \{(0,0), (h',0)\}$ is periodic, which is a contradiction.

Now, let us consider the case where $(h,1)$ is not a period of $\widetilde{T}$ for any $h$ with $\text{ord}(h) = 2$. For any distinct  $(h_1,i_1), (h_2,i_2)$ in $\widetilde{T}$, we have
\[
\big(\widetilde{\Omega} + (h_1,i_1)\big) \cap \big(\widetilde{\Omega} + (h_2,i_2)\big) = \emptyset.
\]
If $i_1 = i_2$, then $(\Omega + (h_1,0)) \cap (\Omega + (h_2,0)) = \emptyset$. For $i_1 \neq i_2$, assuming $i_1 = 0$ and $i_2 = 1$, we obtain
\[
\big(\Omega + (h_1,0)\big) \cap \big(\Omega - (h_0,0) + (h_2,1) + (0,1)\big) = \emptyset,
\]
which results in $\big(\Omega + (h_1,0)\big) \cap \big(\Omega + (h_2 - h_0,0)\big) = \emptyset$. Define
\[
T' = \{(h,0) : (h,0) \in \widetilde{T}\} \cup \{(h - h_0,0) : (h,1) \in \widetilde{T}\}.
\]
It follows that $\Omega + T' = H$.

Since $\widetilde{T}$ is a periodic set, it follows that if $(h,0)$ is a period of $\widetilde{T}$, then $(h,0)$ is also a period of $T'$, leading to a contradiction. Similarly, if $(h,1)$ is a period of $\widetilde{T}$, then $(2h,0)$ becomes a period of $T'$, again resulting in a contradiction.
\end{proof}

\subsubsection{Induction from $S\times\Z_{p^n}$ and $S\times \Z_{p^{n+1}}$}

We now turn to the study of how the PT property behaves under cyclic $p$-power extensions. 
The following lemma shows that if a group of the form $H = S \times \Z_{p^n}$ fails to have the PT property, then the same holds for the natural extension $G = S \times \Z_{p^{n+1}}$. This result allows us to inductively extend the non-PT property along higher $p$-powers.

\begin{lemma}\label{Thm-3.7}Let $S$ be a finite abelian group, let $p$ be a prime and let $H = S \times \Z_{p^n}$ with $n\geq 1$. If $H$ does not have the  PT property, then the group $G = S \times \Z_{p^{n+1}}$ does not have the PT  property.
\end{lemma}
\begin{proof}
Any element in $G$ can be expressed as $(s,i)$, where $s\in S$ and $i\in\mathbb{Z}_{p^{n+1}}$. The elements in $H$ take the form $(s,pi)$.
Let  $(\Omega,T)$ be a tiling pair of $H$. Assume  $\Omega$ is not periodic and $T$ cannot be replaced by a periodic set.

Let  \[K=\big\{(0,0), (0,1-p), \cdots, (0,p-1-p)\big\} \] and
  \[\widetilde{\Omega}=\Omega+K.\] By Lemma \ref{prop:2.1}, by taking $h=(0, (p^n-1)p)$, we deduce that  $\widetilde{\Omega}$  is  non-periodic and  $G=\widetilde{\Omega}+T$ is a factorization.

Suppose  $G$ has the  PT property.  Then there exists a periodic subset $\widetilde{T}$ such that $G=\widetilde{\Omega}+\widetilde{T}$ and $(0,0)\in \widetilde{T} $.
For any distinct  $(s_1,pi_1+j_1),(s_2,pi_2+j_2)\in\widetilde{T}$, $j_1,j_2\in \{0,\cdots,p-1\}$, we have
\begin{align*}
\big( \Omega+K+(s_1,pi_1+j_1) \big) \cap \big( \Omega+K+(s_2,p i_2+j_2) \big)=\emptyset.
\end{align*}
Now we shall  show that
\begin{align}\label{eq:empty}
\big(\Omega+(s_1,pi_1)\big)\cap\big(\Omega+(s_2,pi_2)\big)=\emptyset.
\end{align}

Without loss of generality, we assume that $j_1\leq j_2$. We distinguish two cases.

If $j_1=j_2$, then by the choice of $K$,  Equality \eqref{eq:empty} follows automatically.
If $j_1< j_2$,  then
\[\big(\Omega+(s_1,pi_1+j_1)\big)\cap \big(\Omega+(0,j_1-j_2)+(s_2,pi_2+j_2)\big)=\emptyset,\]
which implies
\[\big(\Omega+(s_1,pi_1)\big)\cap \big(\Omega+(s_2,pi_2)\big)=\emptyset.\]

Let
\[T'=\big\{(s,pi): (s,pi+j)\in\widetilde{T}\text{ for some }0\leq j\leq  p-1 \big\}.\]
It  follows that  $\Omega+T'=H$.

Recall that $\widetilde{T}$ is a periodic set. If $(s,pi)$ is a period of $\widetilde{T}$, then it is also a period of $T'$, leading to a contradiction. If $(s,i)$ is a period of $\widetilde{T}$ for some $i$ such that $p\nmid i$, then  $(ps,pi)\neq (0,0)$  is also a period of $\widetilde{T}$. Consequently, $(ps,pi)$ is a period of $T'$, once again resulting in a contradiction.
\end{proof}
\subsection{Proof of Theorem \ref{theorem-subgroup2}}
 The first statement in Theorem \ref{theorem-subgroup2} follows from Lemmas \ref{lem:productextension} and \ref{Thm-3.7}. 
 The second statement in Theorem \ref{theorem-subgroup2}  follows from  Lemma \ref{lem:productextension}.

\section{Extensions of the groups  with the uniformly periodic tiling property} \label{sectionUPTP}
Recall that a group $G$ possesses the  UPT property if every tile in $G$ is either uniformly periodic or dual uniformly periodic. It is easy to see  that groups exhibiting the UPT property  have the PT property.
\begin{lemma}\label{UPT>PT}
Let $G$ be a finite abelian group. If $G$ has the UPT property, then $G$ has the PT property.
\end{lemma}
\begin{proof}
Let $T$ be a  tile of $G$. Then there exists a tiling complement $\Omega$. Therefore, $T$ belongs to $\mathcal{T}_\Omega$.
If there does not exist a periodic $\Omega'$ such that $(T, \Omega')$ forms a tiling pair, then all elements in $\mathcal{T}_\Omega$ are periodic. Hence, either $T$ is periodic or $T$ has a periodic tiling complement.
\end{proof}
We can construct  more groups with the PT property based on the groups possessing  the UPT property.
 On the other hand, it is essential for the group $G$ to possess the UPT  property to guarantee that $G\times \Z_m$ has the PT property.
\begin{proposition}\label{thm:UPTPEX2}
Let $G$ be a finite abelian group. If there exists a tile $\Omega$ and distinct tiling complements $T_0,T_1, \dots, T_{n-1} \in \mathcal{T}_\Omega$ for some integer $n\geq 2$ such that the $T_i$'s do not share a common period,  each element in $\bigcap_{i=0}^{n-1} \mathcal{T}_{T_i}$ is not periodic, then for any integer $m$ with $m\geq n$ and $\gcd(m,|G|)=1$, the group $G \times \mathbb{Z}_m$ does not have the PT property.
\end{proposition}
\begin{proof} 
Elements of $G\times\mathbb{Z}_m$ can be represented by $(g,i)$ with $g\in G$ and $i\in\mathbb{Z}_m$. Let us define:
\begin{align*}
&B=\{(a,0): a\in \Omega\},\\
&S_i=\{(t,i): t\in T_i\} \text{ for }i=0,1,\dots,n-2,\\
&S_i=\{(t,i): t\in T_{n-1}\} \text{ for }i=n-1,\dots,m-1,\\
&S=\bigcup_{i=0}^{m-1} S_i.
\end{align*}
It is straightforward to verify that $B+S_i=G\times \{i\}$, hence $B+S=G\times\mathbb{Z}_m$.
We claim that $S$ is not periodic and $B$ cannot be replaced by a periodic set.

If $S$ is a periodic set with period $(g,i)$ where $g\neq 0$, then $(g,0)$ is a period of $S$ due to $\gcd(m,|G|)=1$. Consequently, all $S_i$ must have the period $(g,0)$, implying that $T_i$ for $i=0,1,\dots, n-1$ share a common period $g$, contradicting our assumption.  If $S$ is a periodic set with period $(0,i)$, then $T_0=T_{i}$. By the assumption $T_0\neq T_i$ if $i\neq 0$.   We deduce that $i=0$.  Hence $S$ is not periodic.

Suppose $B$ can be replaced by a periodic set $B'$.  Given that $|B'| = |B| = |\Omega|$, and since $|\Omega|$ divides $|G|$ and $\gcd(m, |G|) = 1$, according to Lemma \ref{lemma-replace}, we have $G \times \mathbb{Z}_m = mB' + S$, where $mB'$ is also a periodic set. However, if we view $mB'$ as a subset of $G$, this implies $G = mB' + T_i$, leading to a contradiction.
\end{proof}

The following  corollary is a  direct consequence of Proposition \ref{thm:UPTPEX2}. 
\begin{corollary} \label{NUPT}
	 If a finite abelian group $G$ does not have the  UPT property, then for a  sufficiently large integer $m$ with $\gcd(m,|G|)=1$, the group $G\times\mathbb{Z}_m$ does not  have the PT property.
\end{corollary}

Now, we shall prove Theorem \ref{thm:UPTPEX}].   Remark that the second statement    is a consequence of Corollary \ref{NUPT} and the first statement.
\begin{proof}[Proof of Theorem \ref{thm:UPTPEX}]
  Let $(\Omega,T)$ be a tiling pair of  $G\times\mathbb{Z}_m$ with $0\in \Omega$, $0\in T$, $|\Omega|=m_1n_1$, $|T|=m_2n_2$, $m=m_1m_2$, and $|G|=n=n_1n_2$.
  
If $m_2>1$, according to Lemma~\ref{lemma-replace}, $G\times\mathbb{Z}_m=m_2\Omega+T$.  Write $G\times\Z_m= G\times\Z_{m_1}\times \Z_{m_2}$ and identify $G\times \Z_{m_1}$ with the subset 
$G\times \Z_{m_1}\times \{0\}$. Let $T_0=T\cap(G\times\mathbb{Z}_{m_1})$.
Then  we have \[G\times\mathbb{Z}_{m_1}=m_2 \Omega+T_0.\] 
Let $T'=T_0\times\mathbb{Z}_{m_2}$.  Then,  $G\times\mathbb{Z}_m=m_2 \Omega+T'$. 
Write \[\Omega= \bigcup_{j=0}^{m_2} \Omega_{j}\times\{j\}\]
  with $ \Omega_{j}:= \{(g,i)\in G\times \Z_{m_1}:  (g,i,j)\in \Omega\}$.
  Then, \[m_2 \Omega= \bigcup_{j=0}^{m_2}m_2 \Omega_{j}\times\{0\},\]
  and \[\Omega_{j_1}\cap \Omega_{j_2}= \emptyset, \hbox{ for } j_1\neq j_2.\] 
  By Corollary  \ref{cor-mtimes}, we have \[G\times\mathbb{Z}_{m_1}=  (\bigcup_{j=0}^{m_2}\Omega_{j})+T_0.\]
 Since $(h,i,j)+T'=(h,i,0)+T'$, we have $G\times \Z_m=\Omega+T'$, where $T'$ is a periodic set.

Now, let's assume $m_2=1$, then $|\Omega|=n_1m$, $|T|=n_2$. By Lemma~\ref{lemma-replace}, we get \[G\times\mathbb{Z}_m=\Omega+mT.\] Each element of $G\times\mathbb{Z}_m$ can be denoted as $(g,i)$, where $g\in G$ and $i\in\mathbb{Z}_{m}$. Define \[\Omega_i=\{g: (g,i)\in \Omega\}\] and \[T'=\{g: (g,0)\in mT\}.\] Consequently, $\Omega_i+T'=G$ for all $i\in\mathbb{Z}_{m}$.
Since $G$ has the UPT property,  $T'$ can be replaced by a periodic set, or $\Omega_i$ $(i\in[1,m])$ shares a common period, indicating that $\Omega$ is periodic. This  proves the first statement.

The second statement is a consequence of Corollary \ref{NUPT} and the first statement.  If  $G\times \Z_m$ does not have  the UPT  property,   then  there exists a sufficient large integer $m'$ such that  $\gcd(m',|G|m)=1$ and  $G\times \Z_m \times\Z_{m'}$ does not have the PT property, which contradicts  the first statement.
\end{proof}

\section{ Uniformly periodic tiling property}\label{sec-UPTP-group}

    It has been shown  in Section \ref{sectionUPTP}  that additional groups possessing the PT property can be constructed from those exhibiting the UPT property. This section delves into the characterization of groups with the UPT property, focusing particularly on $p$-groups.
    
 
\subsection{Cyclic groups }
We will show that for a non-periodic tile of a cyclic $p$-group, its tiling complements must be periodic with a common period. 
\begin{proposition}\label{pro3}
  Let $G=\mathbb{Z}_{p^n}$ and let $\Omega$ be a non-trivial tile of  $G$.  If $\Omega$ is not periodic, then all tiling complements are periodic of period $p^{n-1}$. 
\end{proposition}
\begin{proof}
	 Since $G$ has the Haj\'os property and $\Omega$ is not periodic, then $T$  is periodic. Hence, $T$ has a period $p^{n-1}$.
\end{proof}
On the other hand, we show that the group $\mathbb{Z}_{p^2q^2}$ does not have the UPT property.  
\begin{proposition}\label{PropP2Q2}
 Let $G=\mathbb{Z}_{p^2q^2}$.  There exists a tile \( \Omega \) which has tiling complements \(T_1, T_2 \in \mathcal{T}_{\Omega}\) such that $T_1$ and  $T_2$ do  not share a common period and  \(  \mathcal{T}_{T_1} \cap \mathcal{T}_{T_2}\) contains no periodic element. Consequently, the group  $G = \mathbb{Z}_{p^2q^2}$ does not have the UPT property.
\end{proposition}
\begin{proof}
Let $a,b\in\mathbb{Z}_{p^2q^2}$ with $\text{ord}(a)=p^2$ and $\text{ord}(b)=q^2$. Let
 \[A=\{ia:i\in\{0,\cdots, p-1\}\},\quad B=\{ib:i\in\{0,\cdots, q-1\}\},\quad \Omega=A+B.\] It is clear that $\Omega$ tiles $G$ by translation with tiling complement \[T=\{ipa+jqb:i\in\{0,\cdots, p-1\},j\in\{0,\cdots, q-1\}\}.\]

Let $T_1=M_1+N_1$ with
\begin{align*}
		&M_1=\{ipa: i\in\{0,\cdots, p-1\}\},\\
        &N_1=\{iqb: i\in\{0,\cdots, q-2\}\}\cup\{(q-1)qb+a\}, 
\end{align*}
and let $T_2=M_2+N_2$ with
\begin{align*}
		&M_2=\{ipa: i\in\{0,\cdots, p-1\}\}\cup\{(p-1)pa+b\},\\
&N_2=\{iqb: i\in\{0,\cdots, q-1\}\}.		
\end{align*}
Note that 
\begin{align*}
		\Omega+T_1&=(A+B)+(M_1+N_1)\\
		       &=\langle a\rangle+B+N_1\\
          &=\langle a\rangle+B+N_2\\
		       &=G.         
	\end{align*}
Similarly,  $\Omega+T_2=G$. Hence, $T_1,T_2 \in \mathcal{T}_{\Omega}$.  

 Note that $T_1+qb\neq T_1$ and $T_2+pa\neq T_2$. This implies that $T_1$ and $T_2$ do not share a common period.
Assume $\Omega'$ is a periodic subset  such that 
\[G=\Omega'+T_1=\Omega'+T_2.\]

If $pa$ were a period of $\Omega'$, then $pa \in \Omega' - \Omega'$. 
Note that $pa \in M_1 - M_1 \subseteq T_1 - T_1$. 
Thus $pa \in (\Omega' - \Omega') \cap (T_1 - T_1)$, which contradicts the fact that $G = \Omega' + T_2$. 
Hence $pa$ cannot be a period of $\Omega'$. 
Similarly, $qb$ cannot be a period of $\Omega'$. 

In fact, this further implies that no nontrivial linear combination $ipa + jqb$ (with $i,j \neq 0$) can be a period of $\Omega$. 
Otherwise,   $iqpa$ is  a period of $\Omega$,  so would $\ell iqpa$ for every $\ell$ coprime to $p$. 
Choosing $\ell$ such that $\ell iq \equiv 1 \pmod{p}$, we obtain that $pa$ is a period of $\Omega$, which is a contradiction.
 
  Therefore, $\Omega'$ cannot be periodic which implies that $G$ does not have the UPT property.
\end{proof}

\subsection{Rank $2$  $p$-groups }    
We will show that the two typical $p$-groups, $\mathbb{Z}_{p^n}$ and $\mathbb{Z}_p^2$ exhibit distinct tiling with uniform periodicity. For a non-periodic tile of cyclic $p$-groups, its tiling complements must be periodic with a common period. 
On the other hand, each tile in the group $G=\mathbb{Z}_{p}^{2}$ is dual uniformly periodic.
\begin{proposition}\label{pro4}
 Let $\Omega$ be a  tile of the group $G=\mathbb{Z}_{p}^{2}$ with $|\Omega|>1$.   Then, there exists a  periodic $\Omega'$ such that for any $T\in \mathcal{T}_\Omega$, $(\Omega',T) $  forms a tiling pair of $G$.
\end{proposition}
\begin{proof}
Without loss of generality, assume that $(0,0)\in \Omega $.
	Let $(r,s)\in \Omega$ be any nonzero element. Let 
	\[\Omega'=\{i(r,s): i\in[0,p-1]\}.\]
	Note that $G=\Omega+T$ if and only if for any $j$ with $\gcd(j,p)=1$, $G=\Omega+jT$ which is equivalent to \[(\Omega-\Omega)\cap j(T-T)=\{(0,0)\}.\] Then, for any $j$ with $\gcd(j,p)=1$, \[(\Omega'-\Omega')\cap j(T-T)=\{(0,0)\},\] and so $G=\Omega'+T$ for all $T\in \mathcal{T}_\Omega$.
\end{proof}

Now, we shall prove that the group  $\Z_{p^n}\times\Z_{p}$  has the  UPT property for each prime $p$. 
 We first present a useful lemma.
 \begin{lemma}\cite[Lemma 3.2]{Z2022}\label{lem:3.2}
 Let $\Omega \subset G=\Z_{p^n}\times\Z_p$. If $( p^{i_1},a), (p^{i_2},0), \cdots , (p^{i_s},0) \in  \mathcal{Z}_\Omega$ for some
$a \in  \Z_p$ and $0\leq i_1 <i_2 <\cdots <i_s \leq n-1$, then $ p^s |  |\Omega| $.
 \end{lemma}

 \begin{proposition}\label{prop-UPT1}
   For each prime $p$, the group $G=\Z_{p^n}\times\Z_p$ has the UPT property.
 \end{proposition}
 \begin{proof} Assume that $\Omega$ is a tile of  $G$. Then, $|\Omega|\mid |G|=p^{n+1}$. Assume that $|\Omega|=p^t.$ Let \[I_{\Omega}=\{0\leq i\leq n-1: (p^{i},0)\in \mathcal{Z}_\Omega \}.\]  
 By Lemma \ref{lem3p2}, for each $T\in\mathcal{T}_\Omega$, we have  $(p^i,0)\in \mathcal{Z}_T$ for $0\leq i \leq n-1$ and $i\notin I_{\Omega}$.
 By Lemma \ref{lem:3.2}, we have $|I_{\Omega}|=t-1 $ or $t$.
 
 When $|I_{\Omega}|=t-1$, let $A$ be a subset of $\Z_{p^n}$ such that $p^{i}\in \mathcal{Z}_A$ for $i\in I_{\Omega}$. Define \[\Omega'=\bigcup_{j=0}^{p-1}A\times\{j\} \subset G. \]
 It is easy to check that 
 \[\mathcal{Z}_{\Omega'}=G\setminus  \left(\{(0,0)\} \cup\{(p^i,0)\in G: i\notin I_{\Omega}\}\right). \]
Then, $(\Omega',T)$ forms a tiling pair in $G$ for each  tiling complement  $T\in \mathcal{T}_{\Omega}$.
 
 When $|I_{\Omega}|=t$, we distinguish two cases: \[\{(1,j): j\in \Z_p\}\cap \mathcal{Z}_\Omega=\emptyset,   \hbox{ or }\{(1,j): j\in \Z_p\}\cap \mathcal{Z}_\Omega\neq\emptyset.\]

If   $\{(1,j): j\in \Z_p\}\cap \mathcal{Z}_\Omega=\emptyset$, then $\{(1,j): j\in \Z_p\} \subset  \mathcal{Z}_{T}$ for each $T\in \mathcal{T}_{\Omega}$. By Lemma \ref{lem:periodic},   each $T\in \mathcal{T}_{\Omega}$ is periodic, with the same period $(p^{n-1},0)$.

Now we consider the case  $\{(1,j): j\in \Z_p\}\cap \mathcal{Z}_\Omega\neq\emptyset.$  By Lemma \ref{lem:3.2}, we have $(1,0) \in \mathcal{Z}_{\Omega}$, which implies $0\in I_{\Omega}$. 

If $ \{(1,j): j\in \Z_p\}\subset \mathcal{Z}_\Omega$, then by Lemma \ref{lem:periodic}, $\Omega$ is periodic.
Otherwise, $ (1,j_0)\notin \mathcal{Z}_\Omega$ for some $j_0\in \{1,2 \cdots, p-1\}$. Then for any tiling complement $T\subset \mathcal{T}_{\Omega}$,  by Lemma \ref{lem3p2},
$(1,j_0)\in \mathcal{Z}_T.$ 
Let $\beta\in \{1, \cdots, p-1\}$ such that \[1+\beta j_0 \equiv 0\mod p.\] 
Let $J=n-1- I_{\Omega}=\{j_1,j_2,\cdots,j_t\}$ with $0\leq j_1<j_2<\cdots<j_t=n-1$.   Define
\[\Omega'=\{ b(p^{n-1},\beta)+(\sum_{j\in J\setminus\{j_t\} }a_jp^j,0): b, a_j \in \{0,1,\cdots,p-1\}\}. \]
One can check that  $|\Omega|=|\Omega'|$  and 
\[\mathcal{Z}_{\Omega'}=G\setminus\left(\{(0,0), (1,j_0)\}\cup \{(p^i,0)\in G: i\notin I_{\Omega}\}\right). \]
Hence,  $(\Omega',T)$ forms a tiling pair for each $T\in \mathcal{T}_{\Omega}$.
 \end{proof}
 
 On the other hand, the group $\Z_{p^2} \times\Z_{p^2}$ does not have the UPT property.
\begin{proposition}\label{prop:p2p2}
For the group \( G = \mathbb{Z}_{p^2} \times \mathbb{Z}_{p^2} \), there exists a tile \( \Omega \) which has tiling complements \( T_0, T_1, T_2 \in \mathcal{T}_{\Omega} \) such that \(T_0, T_1, T_2\) do  not share a common period and  \( \cap_{j=0}^{2} \mathcal{T}_{T_j} \) contains no periodic element. Consequently, the group  $G = \mathbb{Z}_{p^2} \times \mathbb{Z}_{p^2}$ does not have  the  UPT property.
\end{proposition}
\begin{proof}
Let 
\[
\Omega = \bigcup_{i=0}^{p-1} \{(i, ip), (i, ip+1), \dots, (i, ip+p-1)\}
\]
and let
\begin{align*}
T_0 = \langle (p,1) \rangle,\quad  T_1 = \langle (0,p) \rangle + \langle (p,0) \rangle,\quad T_2 = \langle (1,0) \rangle .
\end{align*}
One can check that
\[
\Omega + T_j = \mathbb{Z}_{p^2} \times \mathbb{Z}_{p^2}.
\]
Note that \(T_0\) has period \((0,p)\), \(T_1\) has periods \((ip,jp)\) for \((i,j) \ne (0,0)\), and \(T_2\) has period \((p,0)\). Hence, $T_0, T_1$ and $ T_2 $ do not have a common period.

Now we show that $\Omega $ cannot be replaced by a periodic set \(\Omega'\). Assume that \(\Omega'\) is periodic and such that \(\Omega' + T_j = \mathbb{Z}_{p^2} \times \mathbb{Z}_{p^2}\) for \(j = 0, 1, 2\). Then, \(\Omega' = \Omega' + (\alpha, \beta)\) for some  \((\alpha, \beta) \neq (0, 0)\). If \(\mathrm{ord}(\alpha, \beta) = p\), then
\[
(\alpha, \beta) \in (T_1 - T_1) \cap (\Omega' - \Omega'),
\]
which leads to a contradiction. If \(\mathrm{ord}(\alpha, \beta) = p^2\), then \((p\alpha, p\beta)\) is a period of \(\Omega'\), which is also a contradiction.
\end{proof}

\subsection{Other $p$-groups with the  uniformly periodic tiling property}
Now, we show that  each tile $\Omega$ in  $\Z_p^3$ can be replaced by a subgroup $\Omega'$ such that $(\Omega',T)$ forms a tiling pair for each $T\in \mathcal{T}_{\Omega}$. 
 We will also show that  the groups $\mathbb{Z}_{p}^{3},\ \mathbb{Z}_{2}^{5}$ have the UPT property, that is each tile in these groups is either uniformly periodic or dual uniformly periodic. Remark that we only need consider the tiles which contains $0$.

 \begin{lemma}\label{lemma--1}
 Let $p$ be a prime, $G=\mathbb{Z}_{p}^n$ and $\Omega$ be a tile of $G$ with $|\Omega|=p^{n-1}$, $0\in \Omega$. Then, there exists a  subgroup  $\Omega'$ such that 
 for any $T\in \mathcal{T}_\Omega$, $(\Omega',T) $  forms a tiling pair of $G$.
 \end{lemma}
 \begin{proof}
 	Since $|\Omega|=p^{n-1}$, then it follows that $\mathcal{Z}_{\Omega}\neq  \mathbb{Z}_{p}^{n} \setminus\{(0,\dots,0)\}$.
 	Let  $(x_1,\dots,x_n)\notin \mathcal{Z}_{\Omega}\cup \{(0,\dots,0)\}$. Then
 	$j(x_1,\dots,x_n)\notin \mathcal{Z}_{\Omega}\cup \{(0,\dots,0)\}$ for $j\in \{1,\cdots,p-1\}$. 
 	Take \[\Omega'=\{(\omega_1,\dots,\omega_3)\in \mathbb{Z}_{p}^{3} : x_1\omega_1+\dots+x_3\omega_3=0\},\]
 	which is a subgroup of order $p^{n-1}$  such that 
 	\[{\rm supp} \widehat{\mathbbm{1}_{\Omega'}}=\{j(x_1,\dots,x_3): j\in[0,p-1]\}.\]
 	By Statement (e) of Lemma \ref{lem3p2}, it follows that $G=\Omega'+T$ for all $T\in \mathcal{T}_\Omega$.
 \end{proof}

\begin{proposition}\label{pro6}
Let $G=\mathbb{Z}_{p}^{3}$ and let $\Omega$ be a non-trivial tile of $G$ with $0\in \Omega$.   Then, there exists a  subgroup  $\Omega'$ such that 
 for any $T\in \mathcal{T}_\Omega$, $(\Omega',T) $  forms a tiling pair of $G$.
 \end{proposition}
 \begin{proof}
Note that  $|\Omega|\mid |G|$.  We distinguish two cases: $|\Omega|=p$, or $|\Omega|=p^2$.
 
 For $|\Omega|=p$, take any  non-zero element $\omega\in \Omega$ and let 
 \[\Omega'=\{i \omega: i\in[0,p-1]\}.\] As in the argument of Proposition \ref{pro4}, $G=\Omega'+T$ for all $T\in \mathcal{T}_\Omega$.
 
 For $|\Omega|=p^2$, the result follows from Lemma~\ref{lemma--1}.
 \end{proof}

%
In the reminder of this section, we shall show that $\mathbb{Z}_{2}^{5}$ and $\mathbb{Z}_4\times\mathbb{Z}_{2}^2$ have the UPT property by using their Haj\'os property.

\begin{proposition}\label{prop15}
	The group $G=\mathbb{Z}_{2}^{5}$ has the UPT property.
\end{proposition}
\begin{proof}
	Assume $\Omega$ is a tile of $G$ with $0\in \Omega$.  Then $|\Omega|\mid|G|$.
	
	If $|\Omega|=2$, then $\Omega$ is a periodic set.
	
	If $|\Omega|=4$, for any $g_1,g_2\in \Omega\backslash\{0\}$, let $\Omega'=\{0,g_1,g_2,g_1+g_2\}$, then $(\Omega'-\Omega')\subset(\Omega-\Omega)$. Hence $(\Omega'-\Omega')\cap(T-T)=\{0\}$ for each $T\in \mathcal{T}_{\Omega}$, and so $G=\Omega'+T$. Thus $\Omega$ can be replaced by a periodic set $\Omega'$.

If $|\Omega|=16$, the result follows from Lemma~\ref{lemma--1}.

Assume $|\Omega|=8$  and $\Omega$ is not periodic. For any $T\in\mathcal{T}_{\Omega}$, $T$ is periodic since $G$ has the Haj\'os property. Note that $|T|=4$. Then each $T\in \mathcal{T}_{\Omega}$ is a subgroup of $G$. 
	Hence, there exists $v_1,v_2,v_3,v_4,v_5\in\mathbb{Z}_{2}^{5}$ such that $\text{rank}(v_1,v_2,v_3,v_4,v_5)=5$, $T=\langle v_4,v_5\rangle$, and $\Omega$ has the following form
	\[\Omega=\big\{\sum_{i=1}^{3}a_iv_i+f_1(a_1,a_2,a_3)v_4+f_2(a_1,a_2,a_3)v_5: a_1,a_2,a_3\in\{0,1\}\big\}.\]
	Let 
	\begin{align*}
		&\Omega_{00}=\{\sum_{i=1}^{3}a_iv_i:f_1(a_1,a_2,a_3)=0,\ f_2(a_1,a_2,a_3)=0\},\\
		&\Omega_{01}=\{\sum_{i=1}^{3}a_iv_i:f_1(a_1,a_2,a_3)=0,\ f_2(a_1,a_2,a_3)=1\},\\
		&\Omega_{10}=\{\sum_{i=1}^{3}a_iv_i:f_1(a_1,a_2,a_3)=1,\ f_2(a_1,a_2,a_3)=0\},\\
		&\Omega_{11}=\{\sum_{i=1}^{3}a_iv_i:f_1(a_1,a_2,a_3)=1,\ f_2(a_1,a_2,a_3)=1\}.
	\end{align*}
	Then $\Omega_{00}\cup\Omega_{01}\cup\Omega_{10}\cup\Omega_{11}=V=\big\{\sum_{i=1}^{3}a_iv_i: a_1,a_2,a_3\in\{0,1\}\big\}$. Without loss of generality, we may assume $|\Omega_{00}|\ge|\Omega_{01}|\ge|\Omega_{10}|\ge|\Omega_{11}|$.
	
	If $|\Omega_{00}|\ge5$, then $V\subset(\Omega_{00}-\Omega_{00})$, and so $\Omega$ can be replaced by the periodic set $V$.
	
	If $|\Omega_{00}|=4$ and $\Omega_{00}=\{0,g_1,g_2,g_1+g_2\}$ for some $g_1,g_2\in V$, let $g\in V\backslash\Omega_{00}$, then at least 2 of $g,g+g_1,g+g_2,g+g_1+g_2$ belong to $\Omega_{01}$. By a direct computation, $\Omega$ can be replaced by the periodic set $\Omega_{00}\cup(\Omega_{00}+g+v_5)$.
	
	If $|\Omega_{00}|=4$ and $\Omega_{00}=\{0,g_1,g_2,g_3\}$ with $\text{rank}(g_1,g_2,g_3)=3$, then at least 2 of $g_1+g_2,g_1+g_3,g_2+g_3,g_1+g_2+g_3$ belong to $\Omega_{01}$. WLOG, assume that $\Omega_{01}\supset\{g_1+g_2,g_1+g_3\}$, then $\Omega$ can be replaced by the periodic set
 \[\{0,g_2,g_3,g_2+g_3\}\cup \left(  \{0,g_2,g_3,g_2+g_3\} +g_1+v_5\right).\]
	
	If $|\Omega_{00}|=3$ and $\Omega_{00}=\{0,g_1,g_2\}$ for some $g_1,g_2\in V$, then $|\Omega_{01}|=3$ or $|\Omega_{01}|=|\Omega_{10}|=2$. Let $g\in V\backslash\langle g_1,g_2\rangle$, then at least 2 of $g,g+g_1,g+g_2,g+g_1+g_2$ belong to $\Omega_{01}$ or $\Omega_{10}$. Hence, $\Omega$ can be replaced by the periodic set 
	\[\{0,g_1,g_2,g_1+g_2\}\cup \left(\{0,g_1,g_2,g_1+g_2\}+ g+v_5\right) \]
or
 \[\{0,g_1,g_2,g_1+g_2\}\cup \left(\{0,g_1,g_2,g_1+g_2\}+ g+v_4\right).\]

	If $|\Omega_{00}|=2$, then $|\Omega_{01}|=|\Omega_{10}|=|\Omega_{11}|=2$. Suppose there exist $i,j,k,l\in\{0,1\}$ and $g\in V$ such that $\Omega_{ij}=g+\Omega_{kl}$ for $(i,j)\neq (k,l)$. WLOG, we assume that $\Omega_{00}=\{0,g_1\}$, $\Omega_{01}=\{g_2,g_2+g_1\}=g_2+\Omega_{00}$. If $\Omega_{10}=\Omega_{11}+g_2$, then $\Omega$ has period $g_2+v_5$. If $\Omega_{10}\ne\Omega_{11}+g_2$, then $\Omega_{10}=\{g_3,g_3+g_2\}$ and $\Omega_{11}=\{g_3+g_1,g_3+g_1+g_2\}$. Hence, $\Omega+g_1+g_2+v_5=\Omega$.

	Now we assume for any $i,j,k,l\in\{0,1\}$ and $g\in V$, $\Omega_{ij}\ne g+\Omega_{kl}$. Without loss of generality, we have
	\begin{align*}
	&	\Omega_{00}=\{0,g_1\}, \Omega_{01}=\{g_2,g_3\},\\
	&\Omega_{10}=\{g_1+g_2,g_2+g_3\}, \Omega_{11}=\{g_1+g_3,g_1+g_2+g_3\},
	\end{align*}
	for some $g_1,g_2,g_3\in V$ with $\text{rank}(g_1,g_2,g_3)=3$. Then, $\Omega$ can be replaced by the periodic set
\[\Omega_{00}\cup  \big(\Omega_{00} +g_2+g_3+v_4\big) \cup \big( \Omega_{00} +g_2+v_5\big) \cup  \big( \Omega_{00}+g_3+v_4+v_5\big). \]
\end{proof}

Before proving the UPT of $\mathbb{Z}_4\times\mathbb{Z}_{2}^2$, we need the following lemmas.
\begin{lemma}\label{lemma--16}
	Let $G=A+B$, where $|A|=2$, $A$ is not periodic. If $A=\{e,a\}$, then $B=B+2a$.
\end{lemma}
\begin{proof}
	Since $G=A+B$, then $G=B\cup(a+B)$. We also have $(a+B)\cup(2a+B)=G$, this implies $B=B+2a$.
\end{proof}

\begin{lemma}\label{lemma-4e}
	Let $G=A+B$, where $|A|=4$, $A$ is not periodic. If $A=\{a_1,a_2,a_3,a_4\}$ and $2a_1=2a_2$, then $a_1-a_2+a_3-a_4\ne0$ or  $a_1-a_2-a_3+a_4\ne0$. Moreover, $a_1-a_2+a_3-a_4+B=B$ and $a_1-a_2-a_3+a_4+B=B$.
\end{lemma}
\begin{proof}
	It is easy to see that $G=a_1+A+B=a_2+A+B$, then
	\begin{align*}
		G=&(2a_1+B)\cup(a_1+a_2+B)\cup(a_1+a_3+B)\cup(a_1+a_4+B)\\
		=&(a_1+a_2+B)\cup(2a_2+B)\cup(a_2+a_3+B)\cup(a_2+a_4+B).
	\end{align*}
	Since $(a_1+B)\cap(a_2+B)=\emptyset$, then $a_1+a_3+B=a_2+a_4+B$ and $a_1+a_4+B=a_2+a_3+B$. That is $a_1-a_2+a_3-a_4+B=B$ and $a_1-a_2-a_3+a_4+B=B$.
	
	If $a_1-a_2+a_3-a_4=0$ and $a_1-a_2-a_3+a_4=0$, then $a_1-a_2$ is a period of $A$, which is a contradiction.
\end{proof}

\begin{lemma}\label{lemma-Z23}
	Let $S\subset\mathbb{Z}_{2}^{3}$ be a set such that for any $s_1,s_2\in S$, we have $s_1+s_2\notin S$. Then there exists non-zero $g_1,g_2\in \mathbb{Z}_{2}^{3}$ such that $g_1,g_2,g_1+g_2\notin S$.
\end{lemma}
\begin{proof}
	If $S$ has only one nonzero element, then it is easy to see that the result follows. Now we assume $S$ has at least two nonzero elements. Let $s_1,s_2\in S$ with $s_1,s_2\ne0$. Let $h\in \mathbb{Z}_{2}^{3}$ such that $\text{rank}(s_1,s_2,h)=3$.
	
	If $h\in S$ or $s_1+s_2+h\in S$, then we can choose $g_1=s_1+s_2$ and $g_2=s_1+h$.
	
	If $s_1+h\in S$ or $s_2+h\in S$, then we can choose $g_1=s_1+s_2$ and $g_2=h$.
	
	This finishes the proof.
\end{proof}

\begin{proposition}\label{prop17}
	The groups $\mathbb{Z}_{4}\times\mathbb{Z}_2\times\mathbb{Z}_2$ has the UPT property.
\end{proposition}
\begin{proof}
	Assume $\Omega$ is a tile of $G$ with $0=(0,0,0)\in \Omega$, then $|\Omega|\mid|G|$.
	
	If $|\Omega|=2$ and $\Omega$ is not periodic, then $\Omega=\{(0,0,0),(i,j,k)\}$ with $i\ne0$  or  $2$. By Lemma~\ref{lemma--16}, $(2i,0,0)$ is a period of $T$ for any $T\in\mathcal{T}_{\Omega}$.

If $|\Omega|=4$ and $\Omega$ is not periodic, then there exist two elements $g_1,g_2\in\Omega$ such that $2g_1=2g_2$. By Lemma~\ref{lemma-4e}, all $T\in\mathcal{T}_{\Omega}$ then share a common period.

	If $|\Omega|=8$ and $\Omega$ is not periodic, then for any $T\in\mathcal{T}_{\Omega}$, $|T|=2$ and $T$ is periodic. If there exists $T_1,T_2,T_3\in\mathcal{T}_{\Omega}$ such that $T_i=\langle g_i\rangle$ and $g_1+g_2=g_3$, then $|\Omega\cap (g+\langle g_1,g_2\rangle)|\le1$ for all $g\in G$. Hence $|\Omega|\le4$, which is a contradiction. By Lemma~\ref{lemma-Z23}, there exists $g_1\ne g_2$ with $\text{ord}(g_1)=\text{ord}(g_2)=2$ such that $\langle g_1\rangle,\langle g_2\rangle,\langle g_1+g_2\rangle\notin\mathcal{T}_{\Omega}$. Hence,  $\Omega$ can be replaced by the periodic set
	\[\langle g_1,g_2\rangle+\{(0,0,0),(1,0,0)\}.\]
\end{proof}

\section{Proof of Theorems \ref{thm-cyc}, \ref{thm-noncyc} and \ref{thm-nonQG}  }\label{sec-proofs}
We aim to provide a complete list of groups with the PT property, but we have not yet achieved this. We have identified all  cyclic groups with the PT property. Additionally, we have found a series of non-cyclic groups with the PT property.

\subsection{Proof of Theorem \ref{thm-cyc} }
It is proved in Proposition \ref{pro3} that the group $\Z_{p^n}$ has the UPT property.  
By  Theorem \ref{thm:UPTPEX},  the group $\Z_{p_{1}^np_2p_3\cdots p_k}$  has the PT property.  

\begin{proposition}\label{theorem-cyclic-good}
 The group $\mathbb{Z}_{p_{1}^np_2p_3\cdots p_k}$ has the PT property.
\end{proposition}

It is known that  $\Z_{p^2q^2}$ and all its subgroups have the Haj\'os property, where $p,q$ are different primes.
We will prove that groups containing $\Z_{p^2q^2}$ as a proper subgroup  do not have the PT property. By Theorem \ref{theorem-subgroup1}, it suffices to show that $\mathbb{Z}_{p^2q^2r}$ and $\mathbb{Z}_{p^3q^2}$  do not have the PT property, see Proposition \ref{theorem-not1} and Proposition \ref{theorem-not2}. 

\begin{proposition}\label{theorem-not1}
  The group $\mathbb{Z}_{p^2q^2r}$  does not have the PT property, where $p,q,r$ are distinct primes.
\end{proposition}
\begin{proof}
It is a direct consequence of Proposition \ref{thm:UPTPEX2} and Proposition  \ref{PropP2Q2}.
\end{proof}

\begin{proposition}\label{theorem-not2}
	The group $\mathbb{Z}_{p^3q^2}$  does not have the PT property, where $p,q$ are distinct primes.
\end{proposition}
\begin{proof}
	Let $a,b\in\mathbb{Z}_{p^3q^2}$ with ${\rm ord}(a)=p^3$ and ${\rm ord}(b)=q^2$. Then $\mathbb{Z}_{p^3q^2}=\langle a,b\rangle$. 
Let 
\[A= \{ipa: i\in\{0,\cdots, p-1\}\}, \quad B=\{ib: i\in\{0,\cdots, q-1\}\}, \quad \Omega=A+B.\]
Define
	\begin{align*}
	    &M=\{ip^2a: i\in \{0,\cdots, p-1\}\},  \quad M_1=\{ip^2a: i\in\{0,\cdots, p-2\}\}\cup\{(p-1)p^2a+b\},\\
		&N=\{iqb: i\in \{0,\cdots, q-1\}\}, \quad  N_1=\{iqb: i\in \{0,\cdots, q-2\}\}\cup\{(q-1)qb+pa\},\\
		&D=\{ia: i \in \{0,\cdots, p-1\}\}, \quad D_1=\{ia: i\in \{1,\cdots, p-1\}\}.
	\end{align*}

Let $T=(M_1+N)\cup(M+N_1+D_1)$. 
Note that  
\[(A+B)+(M_1+N)=A+M_1+\langle b \rangle=A+\langle p^2a, b \rangle= \langle pa, b \rangle\] 
and 
\[(A+B)+(M+N_1)=B+N_1+\langle pa \rangle=B+\langle pa, qb \rangle= \langle pa, b \rangle.\] 
Then we have
	\begin{align*}
		\Omega+T&=(A+B)+((M_1+N)\cup(M+N_1+D_1))\\
		&=(A+B+M_1+N)\cup(A+B+M+N_1+D_1)\\
		&= \langle pa,b\rangle\cup(\langle pa,b\rangle+D_1)\\
		&=\langle pa,b\rangle+D\\
		&=G.          
	\end{align*}

One can check that $T$ and $\Omega$ are non-periodic. If $\Omega$ could be replaced by a periodic set $\Omega'$ with period $p^2a$, then $p^2a \in \Omega' - \Omega'$. Note that $p^2a \in M - M \subseteq T - T$, so $p^2a \in (T - T) \cap (\Omega' - \Omega')$, which contradicts $G = \Omega' + T$. Similarly, $\Omega$ cannot be replaced by a periodic set of period $qb$. Hence, $\mathbb{Z}_{p^3q^2}$  does not have the PT property.
\end{proof}

\subsection{Proof of Theorem \ref{thm-noncyc}} 
Theorem  \ref{thm-noncyc} follows from  Theorems \ref{thm:UPTPEX}, \ref{theorem-subgroup1}, \ref{theorem-subgroup2}  and Propositions \ref{prop-UPT1}, \ref{pro6}, \ref{prop15}, \ref{prop17}.

\subsection{Proof of Theorem \ref{thm-nonQG}}
In this subsection, we show that any group containing either $\mathbb{Z}_{p^2}\times \mathbb{Z}_{p^2}$ or $\mathbb{Z}_{p^2q^2}$ as a proper subgroup typically does not possess the PT property, apart from a few exceptional cases.
\begin{proposition}\label{prop--1}
  The group $\mathbb{Z}_{p^2}^{2}\times\mathbb{Z}_q$ does not have the PT  property, where $p,q$  are distinct primes and $q\geq 3$.
\end{proposition}
\begin{proof}
It is a direct consequence of Proposition \ref{thm:UPTPEX2} and Proposition  \ref{prop:p2p2}.
\end{proof}
\begin{proposition}\label{prop--2}
  The group $\mathbb{Z}_{p^3}\times\mathbb{Z}_{p^2}$ does not have the PT property, where $p$ is an odd prime.
\end{proposition}
\begin{proof}
 Let 
\[
\Omega = \bigcup_{i=0}^{p-1} \{(ip, ip), (ip, ip+1), \dots, (ip, ip+p-1)\}
\]
and let
\begin{align*}
&T_0 = \langle (p^2,1) \rangle,\\
&T_1 = \langle (0,p) \rangle + \langle (p^2,0) \rangle,\\
&T_j = \langle (p,0) \rangle,   \quad 2\leq j \leq p-1.
\end{align*}
Let 
 \[ T=\bigcup_{j=0}^{p-1}T_j+(j,0).\]
One can check that
\[
\Omega + T = \mathbb{Z}_{p^3} \times \mathbb{Z}_{p^2}.
\]
Note that \(T_0\) has period \((0,p)\), \(T_1\) has periods \((ip^2,jp)\) for \((i,j) \ne (0,0)\), and \(T_j\) has period \((p^2,0)\) for  $2\leq j \leq p-1$. Hence, $T_0, T_1, \cdots, T_{p-1}$ do not have a common period, which implies that $T$ is not periodic.

Now we show that $\Omega $ can not be replaced by a periodic set \(\Omega'\). Assume that \(\Omega'\) is a periodic set  such that \(\Omega' + T = \mathbb{Z}_{p^3} \times \mathbb{Z}_{p^2}\). Assume that \(\Omega' = \Omega' + (\alpha, \beta)\) for \((\alpha, \beta) \neq (0, 0)\). If \(\mathrm{ord}(\alpha, \beta) = p\), then
\[
(\alpha, \beta) \in (T_1 - T_1) \cap (\Omega' - \Omega'),
\]
leads to a contradiction. If \(\mathrm{ord}(\alpha, \beta) = p^2\), then \((p\alpha, p\beta)\) is a period of \(\Omega'\), which also leads to a contradiction.
\end{proof}

\begin{lemma}\label{lemma---p}
			Let $p$ be a prime, and let $H$ be a finite abelian group with $|H|\ge 4$, $p\mid |H|$, and $H\ne \mathbb{Z}_{2}^{2}$. Let $g$ be an element of order $p$ in $H$, and let $C_p$ denote the set of all elements of order $p$ in $H$. Define $A = H \setminus \{0,g\}$. Then $C_p \subseteq A-A$.
\end{lemma}
		
		\begin{proof}
			Let $h \in C_p$.  
			
			If $|H|\ge 5$, then $|A|=|H|-2$ and $|A+h|=|H|-2$, both strictly greater than $\tfrac{|H|}{2}$. Hence $A\cap(A+h)\ne \emptyset$, which implies $h\in A-A$.  
			
			If $|H|=4$, then the assumptions $p\mid |H|$ and $H\ne \mathbb{Z}_{2}^{2}$ force $H\cong \mathbb{Z}_4$. In this case $A=\{1,3\}$, so $2\in A-A$.  
		\end{proof}

		\begin{proposition}\label{prop--3}
			Let $p$ be a prime, and let 
			\[
			G=\mathbb{Z}_{p^2a}\times\mathbb{Z}_{p^2b}\times H,
			\]
			with $ab|H|\ge 4$. Assume further that if $p=2$ and $ab=1$, then $H\ne\mathbb{Z}_{2}^{2}$.  
			Then $G$ does not have the PT property.
		\end{proposition}
		
		\begin{proof}
			Define
			\[
			\Omega = \bigcup_{i=0}^{p-1} \{(ia, ipb,0), (ia, (ip+1)b,0), \dots, (ia, (ip+p-1)b,0)\},
			\]
			and
			\begin{align*}
				T_0 &= \langle (0,pb,0) \rangle + \{(0,0,0), (pa,b,0), \dots, ((p-1)pa, (p-1)b,0)\},\\
				T_1' &= \langle (0,pb,0) \rangle + \langle (pa,0,0) \rangle,\\
				T_2' &= \langle (pa,0,0) \rangle + \{(0,0,0), (a,0,0), \dots, ((p-1)a,0,0)\}.
			\end{align*}
			Let 
			\[
			C=\{0,1,\dots,a-1\}\times\{0,1,\dots,b-1\}\times H,
			\]
			which is a complete system of representatives of $G$ modulo $a\mathbb{Z}_{p^2a}\times b\mathbb{Z}_{p^2b}$.  
			Since $ab|H|\ge 4$, and if $p=2$ and $ab=1$ then $H\ne\mathbb{Z}_{2}^{2}$, we can choose $c$ as follows:
			
			If $|H|\ge 4$, then by Lemma~\ref{lemma---p} there exists $h\in H$ such that $(H\setminus\{0,h\})-(H\setminus\{0,h\})$ contains all order-$p$ elements of $H$. Take $c=(0,0,h)$.  
			
			If $|H|<4$ and $a>1$, take $c=(1,0,0)$.  
			
			If $|H|<4$, $a=1$, and $b>1$, take $c=(0,1,0)$.  
			
			Now define
			\begin{align*}
				T_1 &= T_1' + (C\setminus\{(0,0,0),c\}),\\
				T_2 &= T_2' + c,\\
				T &= T_0 \cup T_1 \cup T_2.
			\end{align*}
			
			It is straightforward to verify that  
			\[
			\Omega + T = G.
			\]
			
			Observe that $T_0$ has period $(0,pb,0)$, $T_1'$ has periods $(pai,pbj,0)$ with $(i,j)\ne (0,0)$, and $T_2'$ has period $(pa,0,0)$.  
			Since these three sets do not share a common period, $T$ is not periodic.
			
			\medskip
			We next show that $\Omega$ cannot be replaced by a periodic set $\Omega'$.  
			Suppose $\Omega'$ is periodic and satisfies $\Omega' + T = G$.  
			Assume $\Omega' = \Omega' + (\alpha,\beta,\gamma)$ for some $(\alpha,\beta,\gamma)\ne (0,0,0)$.  
			Since $|\Omega|=p^2$, we must have $\mathrm{ord}(\alpha,\beta,\gamma)\in\{p, p^2\}$.
			
			If $\mathrm{ord}(\alpha,\beta,\gamma)=p$, then
			\[
			(\alpha,\beta,\gamma)\in (T_1-T_1)\cap(\Omega'-\Omega'),
			\]
			a contradiction.  
			
			If $\mathrm{ord}(\alpha,\beta,\gamma)=p^2$, then $(p\alpha,p\beta,p\gamma)$ is a period of $\Omega'$, which again leads to a contradiction.
			
			Hence, $\Omega$ cannot be periodic, and $G$ fails to have the PT property.
		\end{proof}
		
		\begin{proposition}\label{prop:p2q2npt}Let  $p<q$ be distinct primes, and let $H$ be a finite abelian group with $|H|\geq 2$, and $|H|\geq 3$ if $p=2$. Then the group $\Z_{p^2q^2}\times H$ does not have the PT property.
 		\end{proposition}
		
		\begin{proof}
		Let $a,b\in\mathbb{Z}_{p^2q^2}\times\{0\}$ with $\text{ord}(a)=p^2$ and $\text{ord}(b)=q^2$. Define
\[A=\{ia:i\in\{0,\cdots, p-1\}\},\quad B=\{ib:i\in\{0,\cdots, q-1\}\},\quad \Omega=A+B.\]

Now set $T_1=M_1+N_1$, where
\begin{align*}
	&M_1=\{ipa: i\in\{0,\cdots, p-1\}\},\\
	&N_1=\{iqb: i\in\{0,\cdots, q-2\}\}\cup\{(q-1)qb+a\}, 
\end{align*}
and set $T_2=M_2+N_2$, where
\begin{align*}
	&M_2=\{ipa: i\in\{0,\cdots, p-1\}\}\cup\{(p-1)pa+b\},\\
	&N_2=\{iqb: i\in\{0,\cdots, q-1\}\}.		
\end{align*}

Let $C_p'$ denote the set of elements of order $p$ in $H$. 
If $|C_p'|\le1$, define $C_p:=C_p'\cup\{0\}$; otherwise, let $C_p=C_p'$.
Now define
\begin{align*}
	T_1 &= T_1' +(\cup_{c\in C_p}(0,c)),\\
	T_2 &= T_2' + (\cup_{c\in (H\backslash C_p)}(0,c)),\\
	T &= T_1 \cup T_2.
\end{align*}

Observe that
\begin{align*}
	\Omega+T_1'&=(A+B)+(M_1+N_1)\\
	&=\langle a\rangle+B+N_1\\
	&=\langle a\rangle+B+N_2\\
	&=\mathbb{Z}_{p^2q^2}\times\{0\}.         
\end{align*}
Similarly,  $\Omega+T_2'=\mathbb{Z}_{p^2q^2}\times\{0\}$.  Then it is easy to see that  
\[
\Omega + T = G.
\]

Note that $T_1'+qb\neq T_1'$ and $T_2'+pa\neq T_2'$. Thus $T_1'$ and $T_2'$ do not share a common period, and consequently $T$ is non-periodic.

\medskip
We now show that $\Omega$ cannot be replaced by a periodic set $\Omega'$.  
Suppose $\Omega'$ is periodic and satisfies $\Omega' + T = G$.  
Assume $\Omega' = \Omega' + (\alpha,\beta)$ for some $(\alpha,\beta)\ne (0,0)$.  
Since $|\Omega|=pq$, we must have $\mathrm{ord}(\alpha,\beta)\in\{p,q,pq\}$.

If $\mathrm{ord}(\alpha,\beta)=p$, then
\[
(\alpha,\beta)\in (T_1-T_1)\cap(\Omega'-\Omega'),
\]
a contradiction.  

If $\mathrm{ord}(\alpha,\beta)=q$, then
\[
(\alpha,\beta)\in (T_2-T_2)\cap(\Omega'-\Omega'),
\]
a contradiction.  

If $\mathrm{ord}(\alpha,\beta)=pq$, then $(p\alpha,p\beta)$ is a period of $\Omega'$, again leading to a contradiction.

Hence, $\Omega$ cannot be periodic, and therefore $G$ does not have the PT property.
\end{proof}

\begin{proof}[Proof of Theorem \ref{thm-nonQG}]
	Theorem \ref{thm-nonQG} follows from Theorem  \ref{theorem-subgroup1} and  Propositions \ref{prop--1}, \ref{prop--2}, \ref{prop--3} and  \ref{prop:p2q2npt}.
\end{proof}


\section{Rédei property and its implication to PT under subgroup conditions}\label{sec-Red-PT}
This section investigates the interplay between the Rédei property and the PT property.
\begin{proof}[Proof of Theorem \ref{thm:redei}]
	Assume that $\Omega$ tiles $G$ and is not periodic.  
We consider two cases.  

\textbf{(1)} $\langle \Omega \rangle \neq G$.  In this case, suppose $\Omega \subset H$ for some proper subgroup $H \subset G$, and write  
\[
G = \bigcup_{i=1}^{k} (x_i + H),
\]
where $x_1, \dots, x_k \in G$ and the cosets $x_i + H$ are disjoint.  
Since $H$ has the UPT property, there exists a periodic set $T' \subset H$ such that  
\[
H = \Omega + T'.
\]
It follows that $(\Omega,\, \bigcup_{i=1}^{k} (x_i + T'))$ is a tiling pair of $G$.

\medskip
\textbf{(2)} $\langle \Omega \rangle = G$.  Since $G$ has the weak Rédei property, there exists a tiling complement $T$ of $\Omega$ which is contained in some proper subgroup $H \subset G$.  
Similarly, write  
\[
G = \bigcup_{i=1}^{k} (x_i + H),
\]
where $x_1, \dots, x_k \in G$ and the cosets $x_i + H$ are disjoint.  
Define  
\[
\Omega_i = (\Omega \cap (x_i + H)) - x_i, \quad i = 1, \dots, k.
\]
It is clear that each $(\Omega_i, T)$ forms a tiling pair of $H$.  
Since $H$ has the UPT property, either  
\begin{itemize}
    \item all $\Omega_i$ are periodic with the same period, in which case $\Omega$ is periodic (contradiction), or  
    \item $T$ can be replaced by a periodic $T' \subset H$ such that each $(\Omega_i, T')$ is a tiling pair of $H$, which in turn implies that $(\Omega, T')$ is a tiling pair of $G$.
\end{itemize}
\end{proof}

To establish Theorem~\ref{thm:84442}, we first require two auxiliary lemmas.  
They assert that two periodic subsets of size two or four in $\Z_4 \times \Z_4$, even when not sharing a common period, still admit a common periodic tiling complement.

\begin{lemma}\label{lem:Z4Z42}
Let $A,B \subset \Z_4 \times \Z_4$ with $|A| = |B| = 2$.  
Assume that both $A$ and $B$ are periodic but have no common period.  
Then there exists a periodic set $T$ such that both $(A,T)$ and $(B,T)$ are tiling pairs.
\end{lemma}

\begin{proof}
Let $a,b$ be generators of $\Z_4 \times \Z_4$.  
Without loss of generality, suppose $A + 2a = A$ and $B + 2b = B$.  
Thus,
\[
A = \{0,2a\}, 
\qquad 
B = \{0,2b\}.
\]

Define
\[
T = \langle 2(a+b) \rangle + \{e,\, a,\, b,\, a+b\}.
\]
It is straightforward to verify that both $(A,T)$ and $(B,T)$ form tiling pairs.
\end{proof}

\begin{lemma}\label{lem:Z4Z44}
Let $A,B \subset \Z_4 \times \Z_4$ with $|A| = |B| = 4$.  
Assume that both $A$ and $B$ are periodic but have no common period.  
Then there exists a periodic set $T$ such that both $(A,T)$ and $(B,T)$ are tiling pairs.
\end{lemma}

\begin{proof}
Let $a,b$ be generators of $\Z_4 \times \Z_4$.  
Without loss of generality, suppose $A + 2a = A$ and $B + 2b = B$.  
Thus,
\[
A = \langle 2a \rangle + \{e,\, x_1 a + y_1 b\}, 
\qquad 
B = \langle 2b \rangle + \{e,\, x_2 a + y_2 b\},
\]
with $(x_1,y_1) \not\equiv (0,0) \pmod{2}$ and $(x_2,y_2) \not\equiv (0,0) \pmod{2}$.

If $(x_1,y_1) \not\equiv (x_2,y_2) \pmod{2}$, set $(\alpha,\beta) = (x_1+x_2,\, y_1+y_2)$.  
Otherwise, choose $(\alpha,\beta) \in \{(0,1),\, (1,0),\, (1,1)\}$ such that $(\alpha,\beta) \not\equiv (x_1,y_1) \pmod{2}$.

Define
\[
T = \langle 2(a+b) \rangle + \{e,\, \alpha a + \beta b\}.
\]
Again, it can be checked directly that both $(A,T)$ and $(B,T)$ are tiling pairs.
\end{proof}

\begin{proof}[Proof of Theorem \ref{thm:84442}]
Suppose $e \in \Omega \subset G$ is a non-periodic tile of $G$.  
Note first that every proper subgroup of $G$ has the PT property.  
Thus, if $\langle \Omega \rangle \neq G$, then $\Omega$ already admits a periodic tiling complement.  
In particular, if $|\Omega| = 2$, then $\langle \Omega \rangle \neq G$.  
Hence, it suffices to consider the cases $|\Omega| \geq 4$ with $\langle \Omega \rangle = G$.  
We distinguish three cases according to $|\Omega|$.

\medskip
\noindent\textbf{Case (1): $|\Omega| = 4$.}  
Let $T$ be a tiling complement of $\Omega$.  
Since $G$ has the Rédei property, $T$ must lie in a proper subgroup $H \subset G$.  

If $H \ncong \Z_4 \times \Z_4$, then $H$ has the UPT property.  
By the same reasoning as in Case~(2) of Theorem~\ref{thm:redei}, $T$ may be replaced by a periodic tiling complement $T'$.  

Now suppose $H \cong \Z_4 \times \Z_4$.  
If $T$ is periodic, we are done.  
Otherwise, write
\[
G = \bigcup_{i=1}^{2} (x_i + H),
\]
with $x_1, x_2 \in G$ distinct coset representatives, and set
\[
\Omega_i = (\Omega \cap (x_i + H)) - x_i, \quad i = 1,2.
\]
Then each $(\Omega_i, T)$ is a tiling pair of $H$.  
Since $\Z_4 \times \Z_4$ has the Haj\'os property, both $\Omega_1$ and $\Omega_2$ are periodic.  
As $\Omega$ is assumed non-periodic, the two sets must have different periods.  
By Lemma~\ref{lem:Z4Z42}, $T$ can be replaced by a periodic $T'$.  
Thus $\Omega$ has a periodic tiling complement in $G$.

\medskip
\noindent\textbf{Case (2): $|\Omega| = 8$.}  
Let $T$ be a tiling complement of $\Omega$.  
By the Rédei property, $T$ lies in a proper subgroup $H \subset G$.  

If $H \ncong \Z_4 \times \Z_4$, then $H$ has the UPT property, and the same argument as above shows that $T$ may be replaced by a periodic tiling complement.  

If $H \cong \Z_4 \times \Z_4$, then if $T$ is periodic we are done.  
Otherwise, write
\[
G = \bigcup_{i=1}^{2} (x_i + H),
\]
and define $\Omega_i$ as before.  
Each $(\Omega_i, T)$ is a tiling pair of $H$.  
By the Haj\'os property, $\Omega_1$ and $\Omega_2$ are periodic with distinct periods, since $\Omega$ itself is assumed non-periodic.  
By Lemma~\ref{lem:Z4Z44}, $T$ can be replaced by a periodic $T'$.  
Thus $\Omega$ again has a periodic tiling complement in $G$.

\medskip
\noindent\textbf{Case (3): $|\Omega| = 16$.}  
Here $T = \{e,t\}$ must be a tiling complement of $\Omega$.  
Note that $\Omega + 2t = \Omega$.  
If $2t = e$, then $T$ is periodic.  
If $2t \neq e$, then $\Omega$ itself is periodic, contradicting the assumption.  

\medskip
In all cases, $\Omega$ admits a periodic tiling complement.  
This completes the proof.
\end{proof}

\section{Ascending chain structure of tiles}\label{sec-structure}
For tiles in a  group with the PT property, their  structures can be characterized by induction.
Let $A,B$ be two subsets of a finite abelian group $G$ such that $0\in A \cap B$. We define a subset $A\circ_\phi B$ by
\[\{\phi(b)+b\},\]
where $\phi$ is a certain function from $B$ to $A$ with $\phi(0)=0$.
We say that a set $E\subset G$ has  {\bf ascending chain structure} if  there exists a strictly ascending chain of subgroups $0\subset H_1\subset\cdots\subset H_m=G$ and $0\in D_{j}$ ($j=1,\dots,m-1$) is a complete set of coset representatives for $H_{j+1}$ modulo $H_j$, such that \[E=H_1+(D_1\circ_{\phi_1}(D_2+(\cdots\{0\})))\] or \[E=H_1\circ_{\phi_1}(D_1+(D_2\circ_{\phi_2}(\cdots\{0\})))\]
for some function $\phi_i$.

\begin{proof}[Proof of Theorem \ref{theorem-structure}]
	We first assume that $G$ has the PT property and proceed by induction. Suppose the statement holds for all proper subgroups of $G$.  Let $\Omega$ be a tile of $G$.

	If $\Omega$ is periodic, then there exists a subset $\Omega_1\subset\Omega$ and  a subgroup $H_1$ such that $\Omega=\Omega_1+H_1$. Let $\overline{g}=g+H_1$ and $\overline{S}=\{g+H_1: g\in S\}$. Then $\overline{\Omega_1}$ tiles $G/H_1.$ By the induction hypothesis, there exists a strictly ascending chain of subgroups \[\overline{0}\subset \overline{H_2}\subset\cdots\subset \overline{H_m}=G/H_1\] and $0\in E_j$ ($j=2,\dots,m-1$) is a complete set of coset representatives for $\overline{H_{j+1}}$ modulo $\overline{H_j}$, such that \[\overline{\Omega_1}=\overline{H_2}+(E_2\circ_{\phi_1}(E_3+(\cdots\{\overline{0}\})))\] or \[\overline{\Omega_1}=\overline{H_2}\circ_{\phi_1}(E_2+(E_3\circ_{\phi_2}(\cdots\{\overline{0}\})))\]
	for some function $\phi_i$. Lifting back to $G$, there exist subsets $D_j$ of $G$ with $\overline{D_j}=E_j$ and $\overline{D_1}=\overline{H_2}$, such that \[\Omega=H_1+\Omega_1=(H_1+D_1)+(D_2\circ_{\phi_1}(D_3+(\cdots\{\overline{0}\})))\] or \[\Omega=H_1+\Omega_1=H_1+(D_1\circ_{\phi_1}(D_2+(D_3\circ_{\phi_2}(\cdots\{\overline{0}\})))),\] where $H_1+D_1$ forms a subgroup.
	
	If $\Omega$ is not periodic, then there exists a periodic set $T$ such that $G=\Omega+T$. Write $T=T_1+H_1$ with 
 $T_1\subset T$ and  $H_1\le G$. Then $\overline{\Omega}$ tiles $G/H_1$. By induction, we again obtain a strictly ascending chain of subgroups \[\overline{0}\subset \overline{H_2}\subset\cdots\subset \overline{H_m}=G/H_1\] and $0\in E_j$ ($j=2,\dots,m-1$) is a complete set of coset representatives for $\overline{H_{j+1}}$ modulo $\overline{H_j}$, such that \[\overline{\Omega}=\overline{H_2}+(E_2\circ_{\phi_1}(E_3+(\cdots\{\overline{0}\})))\] or \[\overline{\Omega}=\overline{H_2}\circ_{\phi_1}(E_2+(E_3\circ_{\phi_2}(\cdots\{\overline{0}\})))\] 
	for some function $\phi_i$. Lifting to $G$, there exists subsets $D_j$ of $G$ with $\overline{D_j}=E_j$ and $\overline{D_1}=\overline{H_2}$, such that \[\Omega=H_1\circ_{\phi}(D_1+(D_2\circ_{\phi_1}(D_3+(\cdots\{\overline{0}\}))))\] or \[\Omega=H_1\circ_{\phi}(D_1\circ_{\phi_1}(D_2+(D_3\circ_{\phi_2}(\cdots\{\overline{0}\}))))=H_2\circ_{\phi_1}(D_2+(D_3\circ_{\phi_2}(\cdots\{\overline{0}\}))).\]
	
	Now we assume that for any factorization $G=\Omega+T$, there exists a strictly ascending chain of subgroups $0\subset H_1\subset\cdots\subset H_m=G$ and $0\in D_{j}$ ($j=1,\dots,m-1$) is a complete set of coset representatives for $H_{j+1}$ modulo $H_j$, such that \[\Omega=H_1+(D_1\circ_{\phi_1}(D_2+(\cdots\{0\})))\] or \[\Omega=H_1\circ_{\phi_1}(D_1+(D_2\circ_{\phi_2}(\cdots\{0\})))\]
	for some function $\phi_i$.
	If $\Omega$ tiles $G$ and $\Omega$ is non-periodic, then $\Omega$ has the form \[\Omega=H_1\circ_{\phi_1}(D_1+(D_2\circ_{\phi_2}(\cdots\{0\}))).\] Define
	\[T=H_1+(D_1\circ_{\phi_1}(D_2+(\cdots\{0\}))).\]
	It is straightforward to verify that $G=\Omega+T$, and moreover $T$ is periodic. Hence $G$ possesses the PT property.
\end{proof}

As a consequence of Proposition \ref{theorem-structure}, we can characterize the tiles in elementary $p$-groups having the PT property.
\begin{proof}[Proof of Theorem \ref{cor-finitgroup}]
	Assume  $\mathbb{Z}_{p}^n$  has the  PT property.  Then by Proposition~\ref{theorem-structure}, there exists a strictly ascending chain of subgroups $0\subset H_1\subset\cdots\subset H_m=\mathbb{Z}_{p}^n$ and $0\in D_{j}$ ($j=1,\dots,m-1$) is a complete set of coset representatives for $H_{j+1}$ modulo $H_j$ such that $\Omega=H_1\circ_{\phi_1}(D_1+(D_2\circ_{\phi_2}(\cdots\{0\})))$ or $\Omega=H_1+(D_1\circ_{\phi_1}(D_2+(\cdots\{0\})))$ for some function $\phi_i$.
	Then there exist $a_1,a_2,\dots,a_n\in\mathbb{Z}_{p}^n$ such that $H_i=\langle a_1,\dots,a_{s_i}\rangle$ and $s_m=n$. Note that the complete set of coset representatives for $H_{j+1}$ modulo $H_j$ has the form $D_j=\langle a_1,\dots,a_{s_j}\rangle\circ_{\phi_j}\langle a_{s_j+1},\dots,a_{s_{j+1}}\rangle$ for some function $\phi_j$. Hence $\Omega$ has the following form
	\[\Omega=\langle a_1,\dots,a_{s_1}\rangle\circ_{\phi_1}(\langle a_{s_1+1},\dots,a_{s_{2}}\rangle+((\langle a_1,\dots,a_{s_2}\rangle\circ_{\phi_2}\langle a_{s_2+1},\dots,a_{s_{3}}\rangle)\circ_{\psi_1}(\cdots\{0\})))\]
	or
	\[\Omega=\langle a_1,\dots,a_{s_1}\rangle+(\langle a_{s_1+1},\dots,a_{s_{2}}\rangle\circ_{\psi_1}((\langle a_1,\dots,a_{s_2}\rangle\circ_{\phi_1}\langle a_{s_2+1},\dots,a_{s_{3}}\rangle)+(\cdots\{0\})))\]
	for some functions $\phi_i$ and $\psi_i$.
	For the first case, define
	\[T=\langle a_1,\dots,a_{s_1}\rangle+\langle a_{s_2+1},\dots,a_{s_{3}}\rangle+\langle a_{s_4+1},\dots,a_{s_{5}}\rangle+\dots.\]
	For the second case, define
	\[T=\langle a_{s_1+1},\dots,a_{s_2}\rangle+\langle a_{s_3+1},\dots,a_{s_{4}}\rangle+\langle a_{s_5+1},\dots,a_{s_{6}}\rangle+\dots.\]
	Then it is easy to check that $\mathbb{Z}_{p}^n=\Omega+T$ and $T$ is a subgroup of $\mathbb{Z}_{p}^n$.
\end{proof}

\section{PT property implies ``Tile $\Longrightarrow$ Spectral"}\label{sec-PT--TS}
In this section, we prove that  `$T-S$' holds in a group   if  the group and all its subgroups have the PT property.

\begin{proof}[Proof of Theorem~\ref{PT--TS}]
The proof follows an inductive approach. Assuming the statement holds for all subgroups of $G$, let $(\Omega,T)$ be a tiling pair of group $G$. We distinguish between two cases: (1)  $\Omega$ is periodic, and (2) $\Omega$ is not periodic.

{\bf(1) The tile  $\Omega$ is periodic.} Hence,  $\Omega=\Omega+g$, for some $g\in G\setminus\{0\}$. Without loss of generality,  suppose that the order  ${\rm ord}(g)=p$ is a prime number. 
Write  $G=H\times\mathbb{Z}_{p^n}$ with $g=(0,p^{n-1})$ for some $n\geq 1$.  Then, all elements of $G$ can be represented by $(h,t)$, where $h\in H$ and $t\in\mathbb{Z}_{p^n}$, and all characters of $G$ can be represented by $\chi\psi$, where $\chi\in\widehat{H}$ and $\psi\in\widehat{\mathbb{Z}_{p^n}}$.

 Write $\Omega=\Omega'+\{0\}\times p^{n-1}\mathbb{Z}_{p^n}$, for some $\Omega' \subset H\times\mathbb{Z}_{p^{n-1}}\cong H\times\{0,1,\dots,p^{n-1}-1\}$. Then, \[ H\times\mathbb{Z}_{p^{n-1}}=\Omega'+T',\] where $T'=\{(h,t'): (h,t)\in T, t\equiv t'\pmod{p^{n-1}}\}$. Since $G$ and all its subgroups have the PT property, then there exists $\Gamma\subset H\times\mathbb{Z}_{p^{n-1}}$ such that $(\Omega',\Gamma)$ forms a spectral pair. Define 
 \[\Lambda=\{(h,pt+s): (h,t)\in\Gamma, s\in\mathbb{Z}_p\}.\] Then $|\Lambda|=p|\Gamma|$. For two distinct  $(h,pt+s),(h',pt'+s')\in\Lambda$, we have  
 \[\chi_{h-h'}\psi_{pt-pt'+s-s'}(\Omega'+\{0\}\times p^{n-1}\mathbb{Z}_{p^n})=\chi_{h-h'}\psi_{pt-pt'+s-s'}(\Omega')\psi_{pt-pt'+s-s'}(p^{n-1}\mathbb{Z}_{p^n})=0.\]
  which implies $(\Omega,\Lambda)$ forms a spectral pair of $G$.

{\bf(2) The tile  $\Omega$ is not periodic.} 
Since $G$ has the PT property, the tiling complement $T$ can be chosen to be periodic. Similarly, we can write $G = H \times \mathbb{Z}_{p^n}$ for some prime $p$ and some positive integer $n$ and $T = T' + \{0\} \times \{0,p^{n-1},\cdots, (p-1) p^{n-1}\}$ for some $T' \subset H \times \{0,1,\dots,p^{n-1}-1\}.$

 All characters of $G$ can be represented by $\chi\psi$, where $\chi\in\widehat{H}$ and $\psi\in\widehat{\mathbb{Z}_{p^n}}$.
 Let  \[\Omega'=\{(h,t): (h,t')\in \Omega, t\equiv t'\pmod{p^{n-1}}\}.\] It follows that $|\Omega'|=|\Omega|$ and \[H\times\{0,1,\dots,p^{n-1}-1\}=\Omega'+T'.\]
 Hence, there exists $\Gamma\subset H\times\mathbb{Z}_{p^{n-1}}$ such that $(\Omega',\Gamma)$ forms a spectral pair in group $H\times\mathbb{Z}_{p^{n-1}}$. For any distinct  $(h,t),(h',t')\in\Gamma$, $\chi_{h-h'}\psi_{pt-pt'}(\Omega)=\chi_{h-h'}\psi_{pt-pt'}(\Omega')=0$.  Let $\Lambda=\{(h,pt)\in G: (h,t)\in\Gamma\}.$
    Then, $(\Omega, \Lambda)$ forms a spectral pair in $G$.
\end{proof}

\section{Discussion and questions}\label{sec-con}
In this paper, we introduce the PT property for finite abelian groups and investigate which groups possess it. We completely classify the cyclic groups with the PT property. In addition, we identify certain non-cyclic finite abelian groups that do have the PT property, as well as others that do not. An important application of the PT property is that it implies the implication ``Tile $\Longrightarrow$ Spectral". Consequently, determining the full list of groups with the PT property is a problem of significant interest.

Based on Theorems~\ref{thm-cyc}, \ref{thm-noncyc}, \ref{thm-nonQG} and \ref{thm:redei}, it remains unresolved whether the following groups possess the PT property:
\begin{enumerate}
	\item 
	$\mathbb{Z}_{p_1}^{l_1}\times\mathbb{Z}_{p_2}^{l_2}\times\dots\times\mathbb{Z}_{p_k}^{l_k}$, $p_1=2$, $l_1=6$, $l_2\ge1$; or $p_1=2$, $l_1\ge7$; or $p_1\ge3$, $l_1\ge4$; or $l_2\ge2$,
	\item $\mathbb{Z}_{p^t}\times \mathbb{Z}_{p}^s\times\mathbb{Z}_{p_1}^{l_1}\times\mathbb{Z}_{p_2}^{l_2}\times\dots\times\mathbb{Z}_{p_k}^{l_k}$, $s\le1$, $l_1\ge2$; or $s\ge2$,
	\item $\mathbb{Z}_4\times\mathbb{Z}_2\times\mathbb{Z}_{q^2}$,
	\item $\mathbb{Z}_{q^2}^2$,
	\item $\mathbb{Z}_{2q^2}\times\mathbb{Z}_{q^2}$,
	\item $\mathbb{Z}_{9}^2\times\mathbb{Z}_3$,
	\item $\mathbb{Z}_{4}^2\times\mathbb{Z}_2^2$,
\end{enumerate}
where $p,p_1,\dots,p_k$ are distinct primes, $l_1\ge l_2\ge\dots\ge l_k\ge0$, $t\ge2$, $s\ge0$ are integers, and $q\ge3$ is a prime.

Groups possessing the UPT property are known to be useful for constructing groups with PT property. This begs the question: can we obtain a complete list of groups with  UPT property?  In a forthcoming paper, we will prove that $\Z_{p^3}\times\mathbb{Z}_{2}^{2}$ and  $\Z_{p^2}\times\mathbb{Z}_{2}^{3}$ have the UPT property.


We also investigate the structure of tiles in groups with the PT property. As a byproduct, we prove that any non-trivial tile in the elementary $p$-groups $\Z_p^n$ with the PT property admits a periodic tiling complement. For elementary $p$-groups, we show that every subgroup of $\mathbb{Z}_p^3$ ($p\geq 3$), $\Z_{3}^{4}$, and $\Z_{2}^{6}$ possesses the PT property. This leads us to the following question:
\begin{question}\label{PTFP}
  Do  all elementary $p$-groups $\Z_p^n$ have the  PT property?
\end{question}  
Question~\ref{PTFP} appears to be closely related to the so-called ``Periodic Tiling Conjecture'' on $\Z^d$. For a positive integer $d$,  a set $E\subset \Z^d$ is said to be {\bf periodic} if there exists a finite index subgroup $\Lambda \subset  \Z^d$ such that 
$E+\lambda= E $ for each $\lambda\subset \Lambda$. Formally, we require more for the periodicity in $\Z^d$    than in finite abelian groups.
There is a large body of literature on tilings of $\Z^d$ by translations of finite subsets (see \cite{GT21,Sza04}, and references therein). In the case $d=1$, it is known that any tile of $\Z$ by a finite set $\Omega$ is periodic. However, in higher dimensions tiles need not be periodic. For $d > 1$,  Lagarias and Wang  \cite{LW96invent} proposed the  following conjecture.
\begin{thm} If a finite set $\Omega$ tiles  $\Z^d$ by translation then it admits a periodic tiling.
\end{thm}
 The conjecture was  established  for $d = 2$ by Bhattacharya \cite{Bha20} using techniques from ergodic theory.   Moreover, Greenfeld and  Tao
\cite{GT21} established a quantitative version of the two-dimensional periodic tiling conjecture. 
However, in their celebrated paper  \cite{GT24},  Greenfeld and  Tao disproved  periodic tiling conjecture for the spaces of sufficiently large dimension.   

Although the definitions of periodicity in finite groups and in $\Z^d$
differ, both types of periodicity can be used to characterize the structure of tiles.

\end{document}